\documentclass[12pt]{article}
\usepackage{a4wide}
\usepackage{amsthm}
\usepackage{amsfonts}
\usepackage{epsfig}
\def\CC{{\cal C}}
\def\GG{{G}}
\def\RE{{\cal R}}
\def\GE{{\cal G}}
\def\BE{{\cal B}}
\def\ZZ{{\mathbb Z}}
\newtheorem{theorem}{Theorem}
\newtheorem{lemma}[theorem]{Lemma}
\begin{document}
\title{Short Cycle Covers of Cubic Graphs and Graphs with Minimum Degree Three\thanks{The work of the first three authors was partially supported by the grant GACR 201/09/0197.}}
\author{Tom{\'a}{\v s} Kaiser\thanks{Institute for Theoretical Computer Science (ITI) and Department of Mathematics, University of West Bohemia, Univerzitn\'\i{} 8, 306 14 Plze{\v n}, Czech Republic. E-mail: {\tt kaisert@kma.zcu.cz}. Supported by Research Plan MSM 4977751301 of the Czech Ministry of Education.}\and
        Daniel Kr{\'a}l'\thanks{Institute for Theoretical Computer Science, Faculty of Mathematics and Physics, Charles University, Malostransk{\'e} n{\'a}m{\v e}st{\'\i} 25, 118 00 Prague, Czech Republic. E-mail: {\tt kral@kam.mff.cuni.cz}. The Institute for Theoretical Computer Science (ITI) is supported by Ministry of Education of the Czech Republic as project 1M0545.}\and
	Bernard Lidick{\'y}\thanks{Department of Applied Mathematics, Faculty of Mathematics and Physics, Charles University, Malostransk{\'e} n{\'a}m{\v e}st{\'\i} 25, 118 00 Prague, Czech Republic. E-mail: {\tt \{bernard,bim,samal\}@kam.mff.cuni.cz}.}\and
\newcounter{lth}
\setcounter{lth}{4}
	Pavel Nejedl{\'y}$^{\fnsymbol{lth}}$\and
	Robert {\v S}{\'a}mal$^{\fnsymbol{lth}}$\thanks{Department of Mathematics, Simon Fraser University, 8888 University Dr., Burnaby, BC, V5A 1S6, Canada and Department of Applied Mathematics, Faculty of Mathematics and Physics, Charles University, Malostransk{\'e} n{\'a}m{\v e}st{\'\i} 25, 118 00 Prague, Czech Republic. E-mail: {\tt samal-at-kam.mff.cuni.cz}.}}
\date{}
\maketitle
\begin{abstract}
The Shortest Cycle Cover Conjecture of Alon and Tarsi
asserts that the edges of every bridgeless
graph with $m$ edges can be covered by cycles of total length at most
$7m/5=1.400m$. We show that every cubic bridgeless graph has a cycle
cover of total length at most $34m/21\approx 1.619m$ and every
bridgeless graph with minimum degree three has a cycle cover 
of total length at most $44m/27\approx 1.630m$.
\end{abstract}

\section{Introduction}
\label{sect-intro}

Cycle covers of graphs form a prominent topic in graph theory
which is closely related to several deep and open problems. A {\em cycle}
in a graph is a subgraph with all degrees even. A {\em cycle cover}
is a collection of cycles such that each edge is contained in at least
one of the cycles; we say that each edge is {\em covered}.
The Cycle Double Cover Conjecture of Seymour~\cite{bib-seymour79} and
Szekeres~\cite{bib-szekeres73} asserts that every bridgeless
graph $\GG$ has a collection of cycles containing each edge of $\GG$
exactly twice which is called a {\em cycle double cover}.
In fact, it was conjectured by Celmins~\cite{bib-celmins84} and
Preissmann~\cite{bib-preissmann81} that every graph has a cycle cover
comprised of five cycles.

The Cycle Double Cover Conjecture is related to other deep conjectures in graph
theory. In particular, it is known to be implied
by the Shortest Cycle Cover Conjecture of Alon and Tarsi~\cite{bib-alon85+}
which asserts that every bridgeless graph with $m$ edges has a cycle cover of
total length at most $7m/5$. Recall that the {\em length} of a cycle is
the number of edges contained in it and the length of a cycle cover
is the sum of the lengths of its cycles. 
The reduction of the Cycle Double Cover Conjecture
to the Shortest Cycle Cover Conjecture can be found in
the paper of Jamshy and Tarsi~\cite{bib-jamshy92+}.

The best known general result on short cycle covers
is due to Alon and Tarsi~\cite{bib-alon85+} and
Bermond, Jackson and Jaeger~\cite{bib-bermond83+}: every
bridgeless graph with $m$ edges has a cycle cover of total
length at most $5m/3\approx 1.667m$.
There are numerous results on short cycle covers
for special classes of graphs, e.g., graphs with no short cycles,
highly connected graphs or graphs admitting a nowhere-zero $4$-/$5$-flow, 
see e.g.~\cite{bib-fan92,bib-fan94,bib-jackson90,bib-jamshy87+,bib-raspaud91}.
The reader is referred to the monograph of Zhang~\cite{bib-zhang97}
for further exposition of such results where an entire chapter
is devoted to results on the Shortest Cycle Cover Conjecture.

The least restrictive of refinements of the general
bound of Alon and Tarsi~\cite{bib-alon85+} and
Bermond, Jackson and Jaeger~\cite{bib-bermond83+}
are the bounds for cubic graphs.
Jackson~\cite{bib-jackson94}
showed that every cubic bridgeless graph with $m$ edges
has a cycle cover of total length at most $64m/39\approx 1.641m$ and
Fan~\cite{bib-fan94} later showed that every such graph
has a cycle cover of total length at most $44m/27\approx 1.630m$.
In this paper, we strengthen this bound in two different ways. First,
we improve it to $34m/21\approx 1.619m$ for {\em cubic} bridgeless
graphs. And second, we show that the bound $44m/27$
also holds for $m$-edge bridgeless graphs with {\em minimum degree three},
i.e., we extend the result from~\cite{bib-fan94} on cubic graphs
to all graphs with minimum degree three.
As in~\cite{bib-fan94}, the cycle
covers that we construct consist of at most three cycles.

Though the improvements of the original bound of $5m/3=1.667m$
on the length of a shortest cycle cover of an $m$-edge bridgeless graph
can seem to be rather minor, obtaining a bound below $8m/5=1.600m$
for a significant class of graphs might be quite challenging since
the bound of $8m/5$ is implied by Tutte's $5$-Flow Conjecture~\cite{bib-jamshy87+}.

The paper is organized as follows. In Section~\ref{sec-prelim}, we introduce
auxiliary notions and results needed for our main results. We demonstrate
the introduced notation in Section~\ref{sec-intermezzo} where we reprove
the bound $5m/3$ of Alon and Tarsi~\cite{bib-alon85+} and
Bermond, Jackson and Jaeger~\cite{bib-bermond83+} using our terminology.
In Sections~\ref{sec-cubic} and~\ref{sec-mindegree}, we prove our bounds
for cubic graphs and graphs with minimum degree three.

\section{Preliminaries}
\label{sec-prelim}

In this section, we introduce notation and auxiliary concepts used throughout the paper.
We focus on those terms where confusion could arise and
refer the reader to standard graph theory textbooks,
e.g.~\cite{bib-diestel00}, for exposition of other notions.

Graphs considered in this paper can have loops and multiple (parallel) edges.
If $E$ is a set of edges of a graph $\GG$, $\GG\setminus E$ denotes
the graph with the same vertex set with the edges of $E$ removed.
If $E=\{e\}$, we simply write $\GG\setminus e$ instead of $\GG\setminus\{e\}$.
For an edge $e$ of $\GG$, $\GG/e$ is the graph obtained by contracting
the edge $e$, i.e., $\GG/e$ is the graph with the end-vertices of
$e$ identified, the edge $e$ removed and all the other edges,
including new loops and parallel edges, preserved.
Note that if $e$ is a loop, then $\GG/e=\GG\setminus e$.
Finally, for a set $E$ of edges of a graph $\GG$,
$\GG/E$ denotes the graph obtained by contracting all edges contained
in $E$. If $\GG$ is a graph and $v$ is a vertex of $\GG$ of degree two,
then the graph obtained from $\GG$ by {\em suppressing} the vertex $v$
is the graph obtained from $\GG$ by contracting one of the edges incident
with $v$, i.e., the graph obtained by replacing the two-edge path
with the inner vertex $v$ by a single edge.

An {\em edge-cut} in a graph $\GG$ is a set $E$ of edges such that
the vertices of $\GG$ can be partitioned into two sets $A$ and $B$
such that $E$ contains exactly the edges with one end-vertex in $A$ and
the other in $B$. Such an edge-cut is also denoted by $E(A,B)$.
Note that edge-cuts need not be minimal sets of edges whose removal
increases the number of components of $\GG$.
An edge forming an edge-cut of size one is called a {\em bridge} and
graphs with no edge-cuts of size one are said to be {\em bridgeless}.
Note that we do not require bridgeless graphs to be connected.
The quantity $e(X,Y)$ denotes the number of edges with one end-vertex in $X$ and
the other vertex in $Y$; in particular, if $E(A,B)$ is an edge-cut of $G$,
then $e(A,B)$ is its size.
A graph $\GG$ with no edge-cuts of odd size less than $k$ is
said to be {\em $k$-odd-connected}.
For every set $F$ of edges of~$\GG$, edge-cuts in~$\GG/F$ correspond to edge-cuts
(of the same size) in~$\GG$. Therefore, if $\GG$~has no edge-cuts
of size~$k$, then $\GG/F$ also has no edge-cuts of size~$k$.

As said before, a {\em cycle} of a graph $\GG$ is a subgraph of $\GG$
with all vertices of even degree. A {\em circuit} is a connected
subgraph with all vertices of degree two and a {\em $2$-factor}
is a spanning subgraph with all vertices of degree two.

\subsection{Splitting and expanding vertices}
\label{sub-split}

In the proof of our main result, we will need to construct a nowhere-zero
$\ZZ_2^2$-flow of a special type. 
In order to exclude some ``bad'' nowhere-zero flows, we will first modify
the graph under consideration in such a way that some of its edges must get the same
flow value. This goal will be achieved by splitting some of the vertices
of the graph.
Let $\GG$ be a graph, $v$ a vertex of $\GG$ and $v_1$ and $v_2$ some of
the neighbors of $v$ in $\GG$.
The graph $\GG.v_1vv_2$ that is obtained
by removing the edges $vv_1$ and $vv_2$ from $\GG$ and
adding a two-edge path between $v_1$ and $v_2$ (see Figure~\ref{fig-splitting})
is said to be obtained by {\em splitting the vertices $v_1$ and $v_2$ from the vertex $v$}.
Note that if $v_1=v\not=v_2$, i.e., the edge $vv_1$ is a loop,
the graph $\GG.v_1vv_2$ is the graph obtained
from $\GG$ by removing the loop $vv_1$ and subdividing the edge $vv_2$.
Similarly, if $v_1\not=v=v_2$, $\GG.v_1vv_2$ is obtained by removing
the loop $vv_2$ and subdividing the edge $vv_1$.
Finally, if $v_1=v=v_2$, then the graph $\GG.v_1vv_2$ is obtained
from $\GG$ by removing the loops $vv_1$ and $vv_2$ and introducing
a new vertex joined by two parallel edges to $v$.

\begin{figure}
\begin{center}
\epsfbox{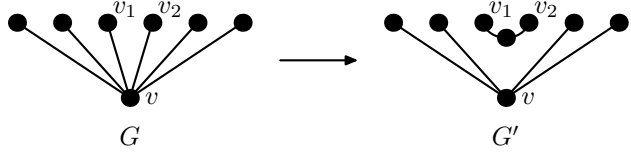}
\end{center}
\caption{Splitting the vertices $v_1$ and $v_2$ from the vertex $v$.}
\label{fig-splitting}
\end{figure}

Classical (and deep) results of 
Fleischner~\cite{bib-fleischner76},
Mader~\cite{bib-mader78} and Lov{\'a}sz~\cite{bib-lovasz76}
assert that it is possible
to split vertices without creating new small edge-cuts.
Let us now formulate one of the corollaries of their results.

\begin{lemma}
\label{lm-split-orig}
Let $\GG$ be a $5$-odd-connected graph.
For every vertex $v$ of $\GG$ of degree four, six or more,
there exist two neighbors $v_1$ and $v_2$ of the vertex $v$ such
that the graph $\GG.v_1vv_2$ is also $5$-odd-connected.
\end{lemma}

Zhang~\cite{bib-zhang02} proved a version of Lemma~\ref{lm-split-orig}
where only some pairs of vertices are allowed to be split off (also see~\cite{bib-szigeti08}
for related results).

\begin{lemma}
\label{lm-split-cyclic}
Let $\GG$ be an $\ell$-odd-connected graph for an odd integer $\ell$.
For every vertex $v$ of $\GG$ with neighbors $v_1,\ldots,v_k$,
$k\not=2,\ell$, there exist two neighbors $v_i$ and $v_{i+1}$ such
that the graph $\GG.v_ivv_{i+1}$ is also $\ell$-odd-connected (indices
are modulo $k$).
\end{lemma}

However, none of these results is sufficient for our purposes
since we need to specify more precisely
which pair of the neighbors of $v$ should be split from $v$.
This is guaranteed by the lemmas we establish in the rest of this
section. Let us remark that Lemma~\ref{lm-split-cyclic} can be obtained
as a consequence of our results.
We start with a modification of the well-known fact that
``minimal odd cuts do not cross'' for the situation where small
even cuts can exist.

\begin{lemma}
\label{lm-uncross}
Let $\GG$ be an $\ell$-odd-connected graph (for some odd integer $\ell$) and
let $E(A_1,A_2)$ and $E(B_1,B_2)$ be two cuts of $\GG$ of size~$\ell$.
Further, let $W_{ij}=A_i\cap B_j$ for $i,j\in\{1,2\}$.
If the sets $W_{ij}$ are non-empty for all $i,j\in\{1,2\}$,
then there exist integers $a$, $b$ and $c$ such that $a+b+c=\ell$ and
one of the following holds:
\begin{itemize}
\item $e(W_{11},W_{12})=e(W_{12},W_{22})=a$,
      $e(W_{11},W_{21})=e(W_{21},W_{22})=b$,
      $e(W_{11},W_{22})=c$ and $e(W_{12},W_{21})=0$, or
\item $e(W_{12},W_{11})=e(W_{11},W_{21})=a$,
      $e(W_{12},W_{22})=e(W_{22},W_{21})=b$,
      $e(W_{12},W_{21})=c$ and $e(W_{11},W_{22})=0$.
\end{itemize}
See Figure~\ref{fig-uncross} for an illustration of the two possibilities.
\end{lemma}

\begin{figure}
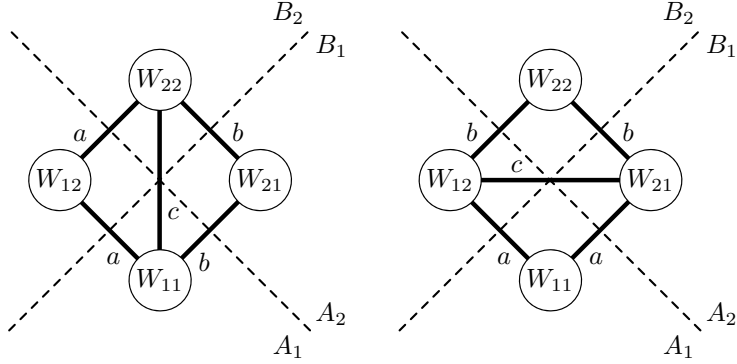

\begin{center}
\epsfbox{scc.28}
\hskip 5mm
\epsfbox{scc.29}
\end{center}
\caption{The two configurations described in the statement of Lemma~\ref{lm-uncross}.}
\label{fig-uncross}
\end{figure}

\begin{proof}
Let $w_{ij}$ be the number of edges with exactly one end-vertex in $W_{ij}$.
Observe that
\begin{equation}
w_{11}+w_{12}=e(A_1,A_2)+2e(W_{11},W_{12})=\ell+2e(W_{11},W_{12})\quad\mbox{and}\label{eq-uncross1}
\end{equation}
$$w_{12}+w_{22}=e(B_1,B_2)+2e(W_{12},W_{22})=\ell+2e(W_{12},W_{22})\;\mbox{.}$$
In particular, one of the numbers $w_{11}$ and $w_{12}$ is even and
the other is odd. Assume that $w_{11}$ is odd. Hence, $w_{22}$ is also odd.
Since $\GG$ is $\ell$-odd-connected, both $w_{11}$ and $w_{22}$
are at least $\ell$.

If $e(W_{11},W_{12})\le e(W_{12},W_{22})$, then 
\begin{eqnarray}
w_{12}& = &e(W_{11},W_{12})+e(W_{12},W_{21})+e(W_{12},W_{22}) \nonumber \\
      & \ge & e(W_{11},W_{12})+e(W_{12},W_{22})
 \ge  2e(W_{11},W_{12})\;\mbox{.}
\label{eq-uncross2}
\end{eqnarray}
The equation (\ref{eq-uncross1}), the inequality (\ref{eq-uncross2}) and
the inequality $w_{11}\ge\ell$ imply that $w_{11}=\ell$.
Hence, the inequality (\ref{eq-uncross2}) is an equality;
in particular, $e(W_{11},W_{12})=e(W_{12},W_{22})$ and
$e(W_{12},W_{21})=0$.
If $e(W_{11},W_{12})\ge e(W_{12},W_{22})$, we obtain the same conclusion.
Since the sizes of the cuts $E(A_1,A_2)$ and $E(B_1,B_2)$ are the same,
it follows that $e(W_{11},W_{21})=e(W_{21},W_{22})$.
We conclude that the graph $\GG$ and the cuts have the structure
as described in the first part of the lemma.

The case that $w_{11}$ is even (and thus $w_{12}$ is odd) leads
to the other configuration described in the statement of the lemma.
\end{proof}

Next, we use Lemma~\ref{lm-uncross} to characterize graphs
where some splittings of neighbors of a given vertex decrease
the odd-connectivity.

\begin{lemma}
\label{lm-split-gener}
Let $\GG$ be an $\ell$-odd-connected graph for an odd integer $\ell\ge 3$,
$v$ a vertex of $\GG$ and $v_1,\ldots,v_k$ some neighbors of $v$.
If every graph $\GG.v_ivv_{i+1}$, $i=1,\ldots,k-1$, contains an edge-cut
of odd size smaller than $\ell$, the vertex set $V(\GG)$ can be partitioned
into two sets $V_1$ and $V_2$ such that $v\in V_1$,
$v_i\in V_2$ for $i=1,\ldots,k$ and
the size of the edge-cut $E(V_1,V_2)$ is $\ell$.
\end{lemma}

\begin{proof}
First observe that
it is enough to prove the statement of the lemma for simple graphs;
if $\GG$ is not simple, then subdivide each edge of $\GG$ and apply
the lemma to the resulting graph in the natural way. The assumption that $\GG$
is simple allows us to avoid unpleasant technical complications in the proof.

The proof proceeds by induction on $k$. The base case of the induction
is that $k=2$.
Let $E(V_1,V_2)$ be an edge-cut of $\GG.v_1vv_2$ of odd size less than $\ell$ and
let $v'$ be the new degree-two vertex of $\GG.v_1vv_2$.
By symmetry, we can assume that $v\in V_1$. We can also assume that $v_1$ and $v'$
are in the same set $V_i$; otherwise, moving $v'$ to the set $V_i$ containing $v_1$
either preserves the size of the cut or decreases the size by two.
If both $v_1\in V_1$ and $v_2\in V_1$, then $E(V_1\setminus\{v'\},V_2)$ as an edge-cut
of $\GG$ has the same size as in $\GG.v_1vv_2$ which
contradicts the assumption that $\GG$ is $\ell$-odd-connected.
If $v_1\in V_1$ and $v_2\in V_2$, then $E(V_1\setminus\{v'\},V_2)$ is also an edge-cut
of $\GG$ of the same size as in $\GG.v_1vv_2$ which is again impossible.

Hence, both $v_1$ and $v_2$
must be contained in $V_2$, and the size of the edge-cut $E(V_1,V_2\setminus\{v'\})$
in $\GG$ is larger by two compared to the size of $E(V_1,V_2)$ in $\GG.v_1vv_2$.
Since $\GG$ has no edge-cuts of size~$\ell-2$, the size of the edge-cut
$E(V_1,V_2)$ in $\GG.v_1vv_2$ is $\ell-2$ and the size of $E(V_1,V_2\setminus\{v'\})$
in $\GG$ is $\ell$.
Hence, $V_1$ and $V_2\setminus\{v'\}$ form the partition of the vertices as
in the statement of the lemma.

\begin{figure}
\begin{center}
\epsfbox{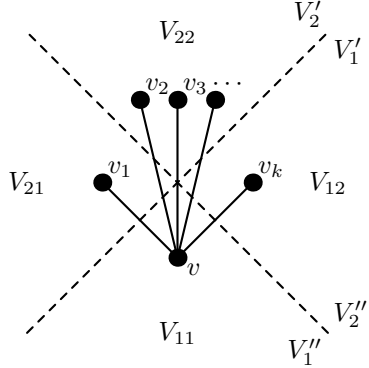}
\end{center}
\caption{Notation used in the proof of Lemma~\ref{lm-split-gener}.}
\label{fig-split-gener}
\end{figure}

We now consider the case that $k>2$. By the induction assumption,
$\GG$ contains a cut $E(V'_1,V'_2)$ of size~$\ell$
such that $v\in V'_1$ and $v_1,\ldots,v_{k-1}\in V'_2$.
Similarly, there is a cut $E(V''_1,V''_2)$ of size~$\ell$ such that $v\in V''_1$ and
$v_2,\ldots,v_k\in V''_2$.
Let $V_{ij}=V'_i\cap V''_j$ for $i,j\in\{1,2\}$ (see
Figure~\ref{fig-split-gener}).
Apply Lemma~\ref{lm-uncross} for the graph $\GG$
with $A_i=V'_i$ and $B_i=V''_i$.
Since $e(V_{11},V_{22})=e(V'_1\cap V''_1,V'_2\cap V''_2)\ge k-2>0$,
the first case described in Lemma~\ref{lm-uncross} applies and
the size of the cut $E(V'_1\cap V''_1,V'_2\cup V''_2)$ is $\ell$.
Hence, the cut $E(V'_1\cap V''_1,V'_2\cup V''_2)$ is a cut of size~$\ell$ in $\GG$.
\end{proof}

We now present a series of lemmas that we need later in the paper.
All these lemmas are simple corollaries of Lemma~\ref{lm-split-gener}.

\begin{lemma}
\label{lm-split4}
Let $\GG$ be a $5$-odd-connected graph, and
let $v$ be a vertex of degree four and $v_1$, $v_2$, $v_3$ and $v_4$
its four neighbors. Then, the graph $\GG.v_1vv_2$ or the graph $\GG.v_2vv_3$
is also $5$-odd-connected graph.
\end{lemma}

\begin{proof}
Observe that the graph $\GG.v_1vv_2$ is $5$-odd-connected
if and only
if the graph $\GG.v_3vv_4$ is $5$-odd-connected.
Lemma~\ref{lm-split-gener} applied for $\ell=5$, the vertex $v$, $k=4$ and
the vertices $v_1$, $v_2$, $v_3$ and $v_4$
yields that the vertices of $\GG$ can be partitioned into two sets $V_1$ and $V_2$
such that $v\in V_1$, $\{v_1,v_2,v_3,v_4\}\subseteq V_2$ and $e(V_1,V_2)=5$.
Hence,
$$e(V_1\setminus\{v\},V_2\cup\{v\})=e(V_1,V_2)-4=5-4=1\;\mbox{.}$$
This contradicts our assumption that $\GG$ has no edge-cuts of size one.
\end{proof}

\begin{lemma}
\label{lm-split6}
Let $\GG$ be a $5$-odd-connected graph, and
let $v$ be a vertex of degree six and $v_1,\ldots,v_6$
its neighbors. At least one of the graphs $\GG.v_1vv_2$, $\GG.v_2vv_3$ and
$\GG.v_3vv_4$ is also $5$-odd-connected.
\end{lemma}

\begin{proof}
Lemma~\ref{lm-split-gener} applied for $\ell=5$, the vertex $v$, $k=4$ and $v_1$, $v_2$, $v_3$ and $v_4$
yields that the vertices of $\GG$ can be partitioned into two sets $V_1$ and $V_2$
such that $v\in V_1$, $\{v_1,v_2,v_3,v_4\}\subseteq V_2$ and $e(V_1,V_2)=5$.
If $V_2$ contains $\sigma$ neighbors of $v$ (note that $\sigma\ge 4$), then
$$e(V_1\setminus\{v\},V_2\cup\{v\})=e(V_1,V_2)-\sigma+(6-\sigma)=11-2\sigma$$
which is equal to $1$ or $3$ contradicting the fact that $\GG$ is
$5$-odd-connected.
\end{proof}

\begin{lemma}
\label{lm-split8}
Let $\GG$ be a $5$-odd-connected graph, and
let $v$ be a vertex of degree $d\ge 6$ and $v_1,\ldots,v_d$ its neighbors.
At least one of the graphs $\GG.v_ivv_{i+1}$, $i=1,\ldots,5$,
is also $5$-odd-connected.
\end{lemma}

\begin{proof}
Since there is no partition of the vertices of $\GG$ into two parts $V_1$ and $V_2$
such that $v\in V_1$, $v_i\in V_2$ for $i=1,\ldots,6$, and $e(V_1,V_2)=5$,
Lemma~\ref{lm-split-gener} applied for $\ell=5$, the vertex $v$, $k=6$ and the vertices $v_i$, $i=1,\ldots,6$
yields the statement of the lemma.
\end{proof}

We will also need the following corollary.

\begin{lemma}
\label{lm-split8spc}
Let $\GG$ be a $5$-odd-connected graph, and
let $v$ be a vertex of degree eight and $v_1,\ldots,v_8$
its neighbors. Suppose that $\GG.v_1vv_2$ is $5$-odd-connected.
At least one of the following graphs is also $5$-odd-connected:
$\GG.v_1vv_2.v_3vv_4$, $\GG.v_1vv_2.v_7vv_8$,
$\GG.v_1vv_2.v_3vv_8.v_4vv_5$ and
$\GG.v_1vv_2.v_3vv_8.v_4vv_6$.
\end{lemma}

\begin{proof}
The degree of the vertex $v$ in $\GG.v_1vv_2$ is six.
By Lemma~\ref{lm-split6}, at least one of the graphs
$\GG.v_1vv_2.v_3vv_4$, $\GG.v_1vv_2.v_3vv_8$ and
$\GG.v_1vv_2.v_7vv_8$ is $5$-odd-connected.

If $\GG.v_1vv_2.v_3vv_8$ is $5$-odd-connected,
we apply Lemma~\ref{lm-split4} for the vertex $v$ and
its neighbors $v_5$, $v_4$, $v_6$ and $v_7$ (in this order).
Hence, the graph $\GG.v_1vv_2.v_3vv_8.v_4vv_5$ (which is isomorphic
to $\GG.v_1vv_2.v_3vv_8.v_6vv_7$) or
the graph $\GG.v_1vv_2.v_3vv_8.v_4vv_6$ is $5$-odd-connected.
\end{proof}

We need one more vertex operation in our arguments---vertex expansion.
If $\GG$ is a graph, $v$ a vertex of $\GG$ and $V_1$ a subset of its neighbors,
then the graph $\GG:v:V_1$
is the graph obtained from $\GG$ by removing the vertex $v$ and introducing
two new vertices $v_1$ and $v_2$, joining $v_1$ to the vertices of $V_1$,
$v_2$ to the neighbors of $v$ not contained in $V_1$, and
adding an edge $v_1v_2$.
We say that $\GG:v:V_1$ is obtained by {\em expanding the vertex $v$
with respect to the set $V_1$}. See Figure~\ref{fig-expand}
for an example. Let us remark that this operation will be applied only
to vertices $v$ incident with no parallel edges.

\begin{figure}
\begin{center}
\epsfbox{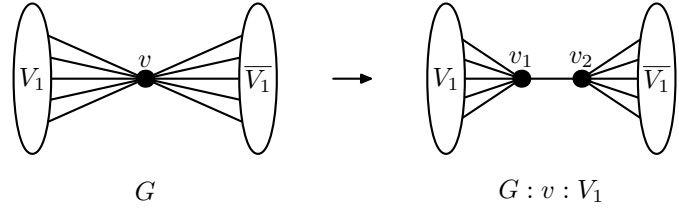}
\end{center}
\caption{An example of the expansion a vertex $v$ with respect to set $V_1$.}
\label{fig-expand}
\end{figure}

The following auxiliary lemma
directly follows from results of Fleischner~\cite{bib-fleischner76}:

\begin{lemma}
\label{lm-expand}
Let $\GG$ be a bridgeless graph and $v$ a vertex of degree four in $\GG$
incident with no parallel edges.
Further, let $v_1$, $v_2$, $v_3$ and $v_4$ be the four neighbors of $v$.
The graph $\GG:v:\{v_1,v_2\}$ or the graph $\GG:v:\{v_2,v_3\}$ is also
bridgeless.
\end{lemma}

\subsection{Special types of $\ZZ_2^2$-flows}
\label{sub-flows}

We start with recalling a classical result of Jaeger
on the existence of nowhere-zero $\ZZ_2^2$-flows.

\begin{theorem}[Jaeger~\cite{bib-jaeger79}]
\label{thm-4flow}
If $\GG$ is a graph that contains no edge-cuts of size one or three,
then $\GG$ has a nowhere-zero $\ZZ_2^2$-flow.
\end{theorem}

In our arguments, we will need to construct $\ZZ_2^2$-flows that
have some additional properties. Some of these properties will be
guaranteed by vertex splittings introduced earlier.
In this section, we establish an auxiliary lemma that guarantees
the existence of a nowhere-zero $\ZZ_2^2$-flow which we were not
able to prove with vertex splittings only. We use the following
notation in Lemma~\ref{lm-flow}: if $\varphi$ is a flow
on a graph $\GG$ and $E$ is a subset of the edges of $G$, then
$\varphi(E)$ is the set of the values of $\varphi$ on $E$.

\begin{lemma}
\label{lm-flow}
Let $\GG$ be a bridgeless graph admitting a nowhere-zero $\ZZ_2^2$-flow.
Assume that 
\begin{itemize}
\item for every vertex $v$ of degree five, there are given two multisets $A_v$ and
      $B_v$ of three edges incident with $v$ such that $|A_v\cap B_v|=2$ (a loop
      incident with $v$ can appear twice in the same set), and
\item for every vertex $v$ of degree six, the incident edges
      are partitioned into three multisets $A_v$, $B_v$ and $C_v$ of size
      two each (every loop incident with $v$ appears twice in the sets $A_v$, $B_v$ and $C_v$,
      possibly in the same set).
\end{itemize}
The graph $\GG$ has a nowhere-zero $\ZZ_2^2$-flow $\varphi$ such that
\begin{itemize}
\item for every vertex $v$ of degree five,
      $|\varphi(A_v)|\ge 2$ and $|\varphi(B_v)|\ge 2$, and
\item for every vertex $v$ of degree six,
      $|\varphi(A_v\cup B_v\cup C_v)|=3$, i.e.,
      the edges incident with $v$ have all the three possible flow values, or
      $|\varphi(A_v)|=1$, or $|\varphi(B_v)|=1$, or $|\varphi(C_v)|=2$.
\end{itemize}
\end{lemma}

\begin{proof}
For simplicity, we refer to edges with the flow value $01$ red, $10$ green and
$11$ blue. Note that each vertex of odd degree is incident with odd numbers of red,
green and blue edges and each vertex of even degree is incident with even numbers
of red, green and blue edges (counting loops twice).
We say that a vertex $v$ of degree five is {\em bad}
if $|\varphi(A_v)|=1$ or $|\varphi(B_v)|=1$,
and it is {\em good}, otherwise. Similarly, a vertex $v$ of
degree six is {\em bad} if $\varphi$ has only two possible flow values at $v$,
$|\varphi(A_v)|=2$, $|\varphi(B_v)|=2$ and $|\varphi(C_v)|=1$;
otherwise, $v$ is good. Choose a $\ZZ_2^2$-flow
$\varphi$ of $G$ with the least number of bad vertices. If there are no bad
vertices, then there is nothing to prove. Assume that there is a bad vertex $v$.

Let us first analyze the case that the degree of $v$ is five. Let $e_1,\ldots,e_5$
be the edges incident with $v$. By symmetry, we can assume that $A_v=\{e_1,e_2,e_3\}$,
$B_v=\{e_2,e_3,e_4\}$, the edges $e_1$, $e_2$ and $e_3$ are red, the edge $e_4$
is green and the edge $e_5$ is blue (see Figure~\ref{fig-flow-bad}). 
We now define a closed trail $W$ in $G$ formed
by red and blue edges. The first edge of $W$ is $e_1$.

\begin{figure}
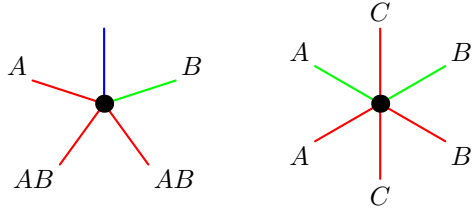

\begin{center}
\epsfbox{scc2.61}
\hskip 5mm
\epsfbox{scc2.62}
\end{center}
\caption{Bad vertices of degree five and six (symmetric cases are omitted). The letters indicate edges contained in the sets $A$, $B$ and $C$.}
\label{fig-flow-bad}
\end{figure}

Let $f=ww'$ be the last edge of $W$ defined so far.
If $w'=v$, then $f$ is one of the edges $e_2$, $e_3$ and
$e_5$ and the definition of $W$ is finished. Assume that $w'\not=v$.
If $w'$ is not a vertex of degree five or six or $w'$ is a bad vertex,
add to the trail $W$ any red or blue edge incident with $w'$ that is not already
contained in $W$.

If $w'$ is a good vertex of degree five, let $f_1,\ldots,f_5$
be the edges incident with $w'$, $A_{w'}=\{f_1,f_2,f_3\}$ and $B_{w'}=\{f_2,f_3,f_4\}$.
If $w'$ is incident with a single red and a single blue edge, leave $w'$
through the other edge that is red or blue. Otherwise, there are three red edges and
one blue edge or vice versa. The next edge $f'$ of the trail $W$ is determined
as follows (note that the role of red and blue can be swapped):
\begin{center}
\begin{tabular}{|c|c|ccccc|}
\hline
Red edges & Blue edge & $f=f_1$ & $f=f_2$ & $f=f_3$ & $f=f_4$ & $f=f_5$ \\
\hline
$f_1$, $f_2$, $f_4$ & $f_3$ & $f'=f_4$ & $f'=f_3$ & $f'=f_2$ & $f'=f_1$ &   N/A \\
$f_1$, $f_2$, $f_4$ & $f_5$ & $f'=f_2$ & $f'=f_1$ &   N/A    & $f'=f_5$ & $f'=f_4$ \\
$f_1$, $f_2$, $f_5$ & $f_3$ & $f'=f_5$ & $f'=f_3$ & $f'=f_2$ &   N/A    & $f'=f_1$ \\
$f_1$, $f_2$, $f_5$ & $f_4$ & $f'=f_2$ & $f'=f_1$ &   N/A    & $f'=f_5$ & $f'=f_4$ \\
$f_1$, $f_4$, $f_5$ & $f_2$ & $f'=f_2$ & $f'=f_1$ &   N/A    & $f'=f_5$ & $f'=f_4$ \\
$f_2$, $f_3$, $f_5$ & $f_1$ & $f'=f_2$ & $f'=f_1$ & $f'=f_5$ &   N/A    & $f'=f_3$ \\
\hline
\end{tabular}
\end{center}
See Figure~\ref{fig-flow-5good} for an illustration of these rules.

\begin{figure}
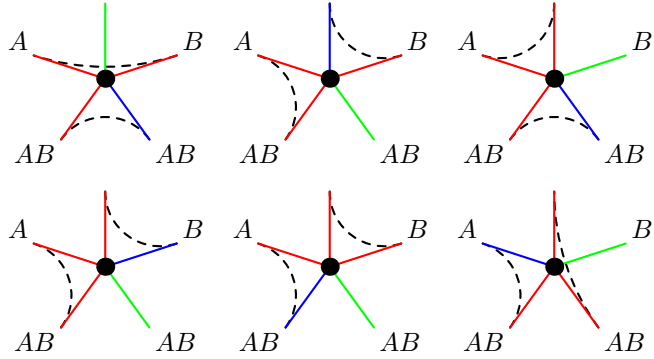

\begin{center}
\epsfbox{scc2.63}
\hskip 2mm
\epsfbox{scc2.64}
\hskip 2mm
\epsfbox{scc2.65}
\end{center}
\begin{center}
\epsfbox{scc2.66}
\hskip 2mm
\epsfbox{scc2.67}
\hskip 2mm
\epsfbox{scc2.68}
\end{center}
\caption{Routing the trail $W$ (indicated by dashed edges) through a good vertex of degree five with three red edges. The letters indicate edges contained in the sets $A$ and $B$. Symmetric cases are omitted.}
\label{fig-flow-5good}
\end{figure}

If $w'$ is a good vertex of degree six, proceed as follows.
If $|\varphi(A_{w'})|=1$ and $f\in A_{w'}$, let the next edge $f'$ of $W$ be the other edge
contained in $A_{w'}$; if $|\varphi(A_{w'})|=1$ and $f\not\in A_{w'}$,
let $f'$ be any red or blue edge not contained in $A_{w'}$ or in $W$. A symmetric
rule applies if $|\varphi(B_{w'})|=1$, i.e., $f'$ is the other edge
of $B_{w'}$ if $f\in B_{w'}$ and $f'$ is a red or blue edge not contained
in $B_{w'}$ or $W$, otherwise.

If $|\varphi(C_{w'})|=2$ and
$f\in C_{w'}$ and the other edge of $C_{w'}$ is red or blue,
set $f'$ to be the other edge of $C_{w'}$; if $f\in C_{w'}$ and the other edge
of $C_{w'}$ is green, choose $f'$ to be any red or blue edge incident with $w'$
that is not contained in $W$. If $f\not\in C_{w'}$ (and $|\varphi(C_{w'})|=2$),
choose $f'$ to be a red or blue
edge incident with $w'$ not contained in $W$ that is also not contained in $C_{w'}$.
If such an edge does not exist, choose $f'$ to be the red or blue edge contained
in $C_{w'}$ (note that the other edge of $C_{w'}$ is green since $w'$ is incident
with an even number of red, green and blue edges).

It remains to consider the case that $w'$ is incident with two edges of each
color and $|\varphi(A_{w'})|=|\varphi(B_{w'})|=2$ and $|\varphi(C_{w'})|=1$.
If $f$ is blue, set $f'$ to be any red edge incident with $w'$ not contained in $W$ and
if $f$ is red, set $f$ to be any such blue edge.
See Figure~\ref{fig-flow-6good} for an illustration of these rules.

\begin{figure}
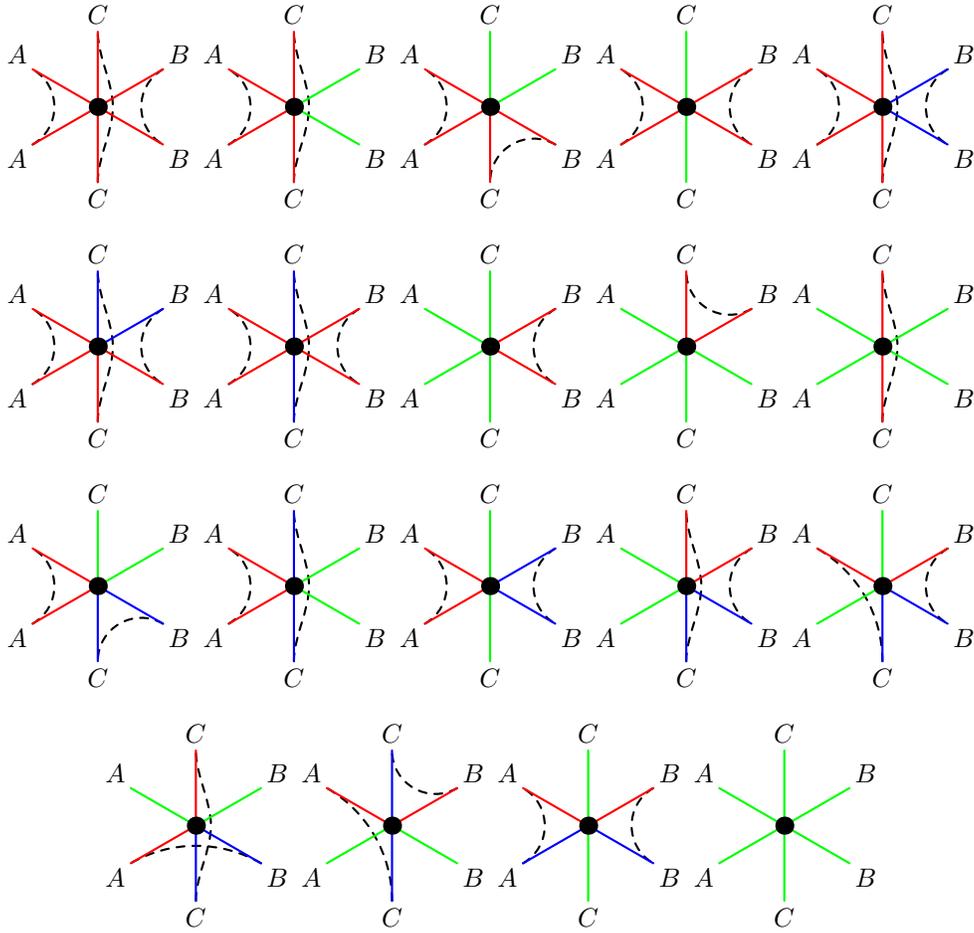

\begin{center}
\epsfbox{scc2.69}
\epsfbox{scc2.70}
\epsfbox{scc2.71}
\epsfbox{scc2.72}
\epsfbox{scc2.73}
\vskip 4mm
\epsfbox{scc2.74}
\epsfbox{scc2.75}
\epsfbox{scc2.76}
\epsfbox{scc2.77}
\epsfbox{scc2.78}
\vskip 4mm
\epsfbox{scc2.79}
\epsfbox{scc2.80}
\epsfbox{scc2.81}
\epsfbox{scc2.82}
\epsfbox{scc2.85}
\vskip 4mm
\epsfbox{scc2.86}
\epsfbox{scc2.83}
\epsfbox{scc2.84}
\epsfbox{scc2.87}
\end{center}
\caption{Routing the trail $W$ (indicated by dashed edges) through a good vertex of degree six. The letters indicate edges contained in the sets $A$, $B$ and $C$. Symmetric cases are omitted.}
\label{fig-flow-6good}
\end{figure}

The definition of the trail $W$ is now finished.
Swap the red and blue colors on $W$.
It is straightforward to verify that all good vertices remain good and the vertex $v$
becomes good (see Figures~\ref{fig-flow-bad}--\ref{fig-flow-6good}).
In particular, the number of bad vertices is decreased which contradicts
the choice of $\varphi$.

Assume that there is a bad vertex $v$ of degree six, i.e.,
the colors of the edges of $A_v$ are distinct, the colors of the edges of $B_v$
are distinct, the colors of the edges of $C_v$ are the same and
not all the flow values are present at the vertex $v$ (see Figure~\ref{fig-flow-bad}).
By symmetry, we can assume that the two edges of $A_v$ are red and green, the two edges
of $B_v$ are also red and green, and the two edges of $C_v$ are both red (recall
that the vertex $v$ is incident with even numbers of red, green and blue edges).
As in the case of vertices of degree five, we find a trail formed by red and
blue edges and swap the colors of the edges on the trail.
The first edge of the trail is any red edge incident with $v$ and the trail $W$
is finished when it reaches again the vertex $v$. After swapping red and blue colors
on the trail $W$, the vertex $v$ is incident with two edges of each of the three
colors. Again, the number of bad vertices has been decreased which contradicts
our choice of the flow $\varphi$.
\end{proof}

\subsection{Rainbow Lemma}
\label{sub-rainbow}

In this section, we present a generalization of an auxiliary lemma
referred to as the Rainbow Lemma.

\begin{lemma}[Rainbow Lemma]
\label{lm-rainbow}
Every cubic bridgeless graph $\GG$ contains
a $2$-factor $F$ such
that the edges of $\GG$ not contained in $F$ can be colored with three
colors, red, green and blue in the following way:
\begin{itemize}
\item every even circuit of $F$ contains an even number of vertices
      incident with red edges, an even number of vertices incident
      with green edges and an even number number of vertices
      incident with blue edges, and
\item every odd circuit of $F$ contains an odd number of vertices
      incident with red edges, an odd number of vertices incident
      with green edges and an odd number number of vertices
      incident with blue edges.
\end{itemize}
\end{lemma}

\noindent In the rest, a $2$-factor $F$ with an edge-coloring satisfying
the constraints given in Lemma~\ref{lm-rainbow} will be called
a {\em rainbow $2$-factor}.
Rainbow $2$-factors implicitly appear in, e.g.,
\cite{bib-fan94,bib-kral+,bib-macajova05+}, and are related
to the notion of parity $3$-edge-colorings
from the Ph.D.~thesis of Goddyn~\cite{bib-goddyn88}.
The lemma follows from Theorem~\ref{thm-4flow} and the folklore statement
that every cubic bridgeless graph has a $2$-factor $F$ such that $\GG/F$
is $5$-odd-connected (see~\cite{bib-kral06+,bib-zhang97}).
The following strengthening of this statement appears in~\cite{bib-fan90}
which we use in the proof of Lemma~\ref{lm-rainbow-mindegree}.

\begin{lemma}
\label{lm-rainbow-weighted}
Let $\GG$ be a bridgeless cubic graph with edges assigned weights and
let $w_0$ be the total weight of all the edges of $\GG$.
The graph $\GG$ contains a $2$-factor $F$ such that
$\GG/F$ is $5$-odd-connected and
the total weight of the edges of $F$ is at most $2w_0/3$.
\end{lemma}

For the proofs of our bounds, we need two strengthenings of the Rainbow Lemma.
Some additional notation must be presented.
The {\em pattern} of a circuit $C=v_1\ldots v_k$ of a rainbow $2$-factor $F$
is $X_1\ldots X_k$ where $X_i$ is the color of the edge
incident with the vertex $v_i$;
we use R to represent the red color,
G the green color and B the blue color. Two patterns are said to
be {\em symmetric} if one of them can be obtained from the other
by a rotation, a reflection and/or a permutation of the red, green and
blue colors. For example, the patterns RRGBGB and RBRBGG are symmetric
but the patterns RRGBBG and RRGBGB are not.
A pattern $P$ is {\em compatible} with a pattern $P'$
if $P'$ is obtained from $P$
by replacement of some of the colors with the letter x (which represents a wild-card);
note that rotating and reflecting $P'$ are not allowed.
For example, the pattern RGRGBBGG is compatible with RBRxxxBx.

The first modification of the Rainbow Lemma is the following.

\begin{lemma}
\label{lm-rainbow-cubic}
Every cubic bridgeless graph $\GG$ contains
a rainbow $2$-factor $F$ such that
\begin{itemize}
\item no circuit of the $2$-factor $F$ has length three,
\item every circuit of length four has a pattern symmetric to RRRR or RRGG, and
\item every circuit of length eight has a pattern symmetric to one of
      the following 16 patterns: \\
      RRRRRRRR, RRRRRRGG, RRRRGGGG, RRRRGGBB, \\ RRGGRRGG, RRGGRRBB,
      RRRRGRRG, RRRRGBBG, \\ RRGGRGGR, RRGGRBBR, RRGGBRRB,
      RRRRGRGR, \\ RRRGBGBR, RRGRGRGG, RRGRBRBG and RRGGBGBG.
\end{itemize}
\end{lemma}

\begin{proof}
Let $F$ be a $2$-factor of $\GG$ such that the graph $\GG/F$ is $5$-odd-connected.
Since $\GG$ is cubic and $\GG/F$ $5$-odd-connected,
the $2$-factor $F$ contains no circuits of length three.
Let $H_0$ be the graph obtained
from the graph $\GG/F$ by subdividing each edge.
Clearly, $H_0$ is also $5$-odd-connected.
This modification of $\GG/F$ to $H_0$
is needed only to simplify our arguments later in the proof since it guarantees
that the graph is simple and makes it more convenient to apply
Lemmas~\ref{lm-split4}--\ref{lm-split8spc}.

In a series of steps,
we iteratively modify the graph $H_0$ to graphs $H_1$, $H_2$, etc.
All the graphs $H_1$, $H_2$, \dots will be simple and
$5$-odd-connected.

If the graph $H_i$ contains a vertex $v$ of degree four, then the vertex $v$
corresponds to a circuit $C$ of length four in $F$. Let $v_1$, $v_2$, $v_3$ and $v_4$
be the neighbors of $v$ in the order in which the edges $vv_1$, $vv_2$, $vv_3$ and $vv_4$
correspond to edges incident with the circuit $C$.
The graph $H_{i+1}$ is a $5$-odd-connected graph among $H_i.v_1vv_2$ and $H_i.v_2vv_3$
(by Lemma~\ref{lm-split4} at least one of them is $5$-odd-connected);
see Figure~\ref{fig-rainbow-cubic-1}.
No new vertices of degree four or eight were introduced at this stage.

\begin{figure}
\begin{center}
\epsfbox{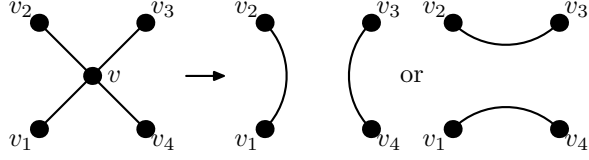}
\end{center}
\caption{Reduction of a vertex of degree four in a graph $H_i$ in the proof of
Lemma~\ref{lm-rainbow-cubic}. The newly created vertices of degree two are not
drawn in the figure.}
\label{fig-rainbow-cubic-1}
\end{figure}

If the graph $H_i$ contains a vertex $v$ of degree eight, we proceed as follows.
The vertex $v$ corresponds to a circuit $C$ of $F$ of length eight;
let $v_1$, \dots, $v_8$
be the neighbors of $v$ in $H_i$ in the order in that they correspond to the edges
incident with the circuit $C$. By Lemma~\ref{lm-split8}, we can assume that the graph
$H_i.v_1vv_2$ is $5$-odd-connected (for a suitable choice of the cyclic
rotation of the neighbors of $v$).

By Lemma~\ref{lm-split8spc},
at least one of the following graphs is $5$-odd-connected:
$H_i.v_1vv_2.v_3vv_4$, $H_i.v_1vv_2.v_7vv_8$, $H_i.v_1vv_2.v_3vv_8.v_4vv_5$ and
$H_i.v_1vv_2.v_3vv_8.v_4vv_6$. If the graph $H_i.v_1vv_2.v_3vv_4$ 
is $5$-odd-connected, we then apply Lemma~\ref{lm-split4} to the graph
$H_i.v_1vv_2.v_3vv_4$ and conclude that the graph $H_i.v_1vv_2.v_3vv_4.v_5vv_6$ or
the graph $H_i.v_1vv_2.v_3vv_4.v_6vv_7$ is $5$-odd-connected.
Since the case that the graph $H_i.v_1vv_2.v_7vv_8$ is $5$-odd-connected
is symmetric to the case of the graph $H_i.v_1vv_2.v_3vv_4$,
it can be assumed that (at least) one of the following four graphs is $5$-odd-connected:
$H_i.v_1vv_2.v_3vv_4.v_5vv_6$, $H_i.v_1vv_2.v_3vv_4.v_6vv_7$,
$H_i.v_1vv_2.v_3vv_8.v_4vv_5$ and $H_i.v_1vv_2.v_3vv_8.v_4vv_6$.
Let $H_{i+1}$ be a graph among these graphs that
is $5$-odd-connected (see Figure~\ref{fig-rainbow-cubic-2}).

\begin{figure}
\begin{center}
\epsfbox{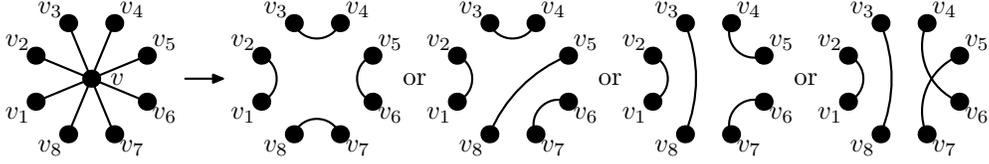}
\end{center}
\caption{Reduction of a vertex of degree eight in a graph $H_i$ in the proof of Lemma~\ref{lm-rainbow-cubic}. The newly created vertices of degree two are not
drawn in the figure.}
\label{fig-rainbow-cubic-2}
\end{figure}

We eventually reach a 5-odd-connected graph $H_k$
with no vertices of degree four or eight.
By Theorem~\ref{thm-4flow}, the graph $H_k$ has a nowhere-zero $\ZZ_2^2$-flow.
This flow yields a nowhere-zero $\ZZ_2^2$-flow
in $H_{k-1},\ldots, H_0$ and eventually in $G/F$.
Note that the pairs of edges split away from a vertex
are assigned the same flow value.

The nowhere-zero $\ZZ_2^2$-flow of $G/F$ gives the coloring of the edges
with red, green and blue.
Since the colors of the edges of $\GG/F$ correspond
to a nowhere-zero $4$-flow of $\GG/F$,
$F$ is a rainbow $2$-factor with respect to this edge-coloring.

We now verify that the edge-coloring of $\GG/F$ also satisfies
the additional constraints given in the statement.
Let us start with the first constraint, and
let $C$ be a circuit of $F$ of length four, $v$ the vertex of $\GG/F$ corresponding
to $C$ and $e_1$, $e_2$, $e_3$ and $e_4$ the four (not necessarily distinct)
 edges leaving $C$ in $\GG$.
In $H_0$, the edge $e_i$ corresponds to
an edge $vv_i$ for a neighbor $v_i$ of $v$.
During the construction of $H_k$, either the vertices $v_1$ and $v_2$ or
the vertices $v_2$ and $v_3$ are split away from $v$.
In the former case,
the colors of the edges $e_1$ and $e_2$ are the same and
the colors of the edges $e_3$ and $e_4$ are the same;
in the latter case,
the colors of the edges $e_1$ and $e_4$ and
the colors of the edges $e_2$ and $e_3$ are the same.
In both cases, the pattern of $C$ is symmetric to RRRR or RRGG.

Let $C$ be a circuit of $F$ of length eight, $v$ the vertex of $\GG/F$ corresponding
to $C$ and $e_1$, \dots, $e_8$ the eight edges leaving $C$ in $\GG$ (note that
some of the edges $e_1$, \dots, $e_8$ can be the same). Let $c_i$ be the color
of the edge $e_i$. Based on the splitting,
one of the following four cases (up to symmetry) applies:
\begin{enumerate}
\item $c_1=c_2$, $c_3=c_4$, $c_5=c_6$ and $c_7=c_8$,
\item $c_1=c_2$, $c_3=c_4$, $c_5=c_8$ and $c_6=c_7$,
\item $c_1=c_2$, $c_3=c_8$, $c_4=c_5$ and $c_6=c_7$, and
\item $c_1=c_2$, $c_3=c_8$, $c_4=c_6$ and $c_5=c_7$.
\end{enumerate}
In the first case, the pattern of the circuit $C$ is symmetric to
RRRRRRRR, RRRRRRGG, RRRRGGGG, RRRRGGBB, RRGGRRGG or RRGGRRBB.
In the second case, the pattern of $C$ is symmetric to
RRRRRRRR, RRRRRRGG, RRRRGGGG, RRRRGGBB, RRRRGRRG, RRRRGBBG,
RRGGRGGR, RRGGRBBR or RRGGBRRB.
The third case is symmetric to the second one (see Figure~\ref{fig-rainbow-cubic-2}).
In the last case, the pattern of $C$ is symmetric to
RRRRRRRR, RRRRRRGG, RRRRGRGR, RRRRGRRG, RRRRGGGG, RRGRGRGG,
RRRGBGBR, RRGRBRBG, RRGGBGBG or RRRRGBBG.
In all the four cases, the pattern of $C$ is one of the patterns
listed in statement of the lemma.
\end{proof}

We now present the second modification of the rainbow lemma.

\begin{lemma}
\label{lm-rainbow-mindegree}
Let $\GG$ be a bridgeless cubic graph with edges assigned non-negative
integer weights and $w_0$ be the total weight of the edges.
In addition, suppose that no two edges with weight zero have
a vertex in common. The graph $\GG$ contains a rainbow $2$-factor $F$ such
that the total weight of the edges of $F$ is at most $2w_0/3$.
Moreover, the patterns of circuits with four edges of weight one
are restricted as follows.
Every circuit $C=v_1\ldots v_k$ of $F$ that consists of
four edges of weight one and at most four edges of weight zero (and
no other edges) has a pattern:
\begin{itemize}
\item compatible with RRxx or xRRx
      if $C$ has no edges of weight zero (and thus $k=4$),
\item compatible with RxGxx or RRRGB
      if the only edge of $C$ of weight zero is $v_4v_5$ (and thus $k=5$),
\item compatible with xxRRxx, xxxxRR, xxRGGR or xRxGGR
      if the only edges of $C$ of weight zero are $v_3v_4$ and $v_5v_6$ (and thus $k=6$),
\item not compatible with RRGRRG, RRGRGR, RGRRRG or RGRRGR
      if the only edges of $C$ of weight zero are $v_2v_3$ and $v_5v_6$ (and thus $k=6$),
\item compatible with xRRxxxx, xxxRRxx, xxxxxRR, xRGxxRB, xRGxxBR, xRGxxGB, xRGxxBG, xxxRGRG or xxxRGGR
      if the only edges of $C$ of weight zero are $v_2v_3$, $v_4v_5$ and $v_6v_7$ (and thus $k=7$), and
\item compatible with RRxxxxxx, xxRRxxxx, xxxxRRxx, xxxxxxRR, RGGRxxxx, xxRGGRxx, xxxxRGGR or GRxxxxRG
      if the edges $v_1v_2$, $v_3v_4$, $v_5v_6$ and $v_7v_8$ of $C$ have
      weight zero (and thus $k=8$).
\end{itemize}
\end{lemma}

\begin{proof}
Let $F$ be the $2$-factor as in Lemma~\ref{lm-rainbow-weighted} and
$M$ the complementary perfect matching.
We modify the graph $H=\GG/F$ in such a way that
an application of Lemma~\ref{lm-flow} will yield
a $\ZZ_2^2$-flow that yields an edge-coloring satisfying the conditions
from the statement of the lemma.
Let $w$ be a vertex of $H$ corresponding to a circuit $v_1\ldots v_k$ of $F$
consisting of four edges with weight one and some edges with weight zero, and
let $e_i$ be the edge of $M$ incident with $v_i$.
Finally, let $w_i$ be the neighbor of $w$ in $H$ that corresponds
to the circuit containing the other end-vertex of the edge $e_i$.
The graph $H$ is modified as follows (see Figure~\ref{fig-rainbow-mindegree}):

\begin{figure}
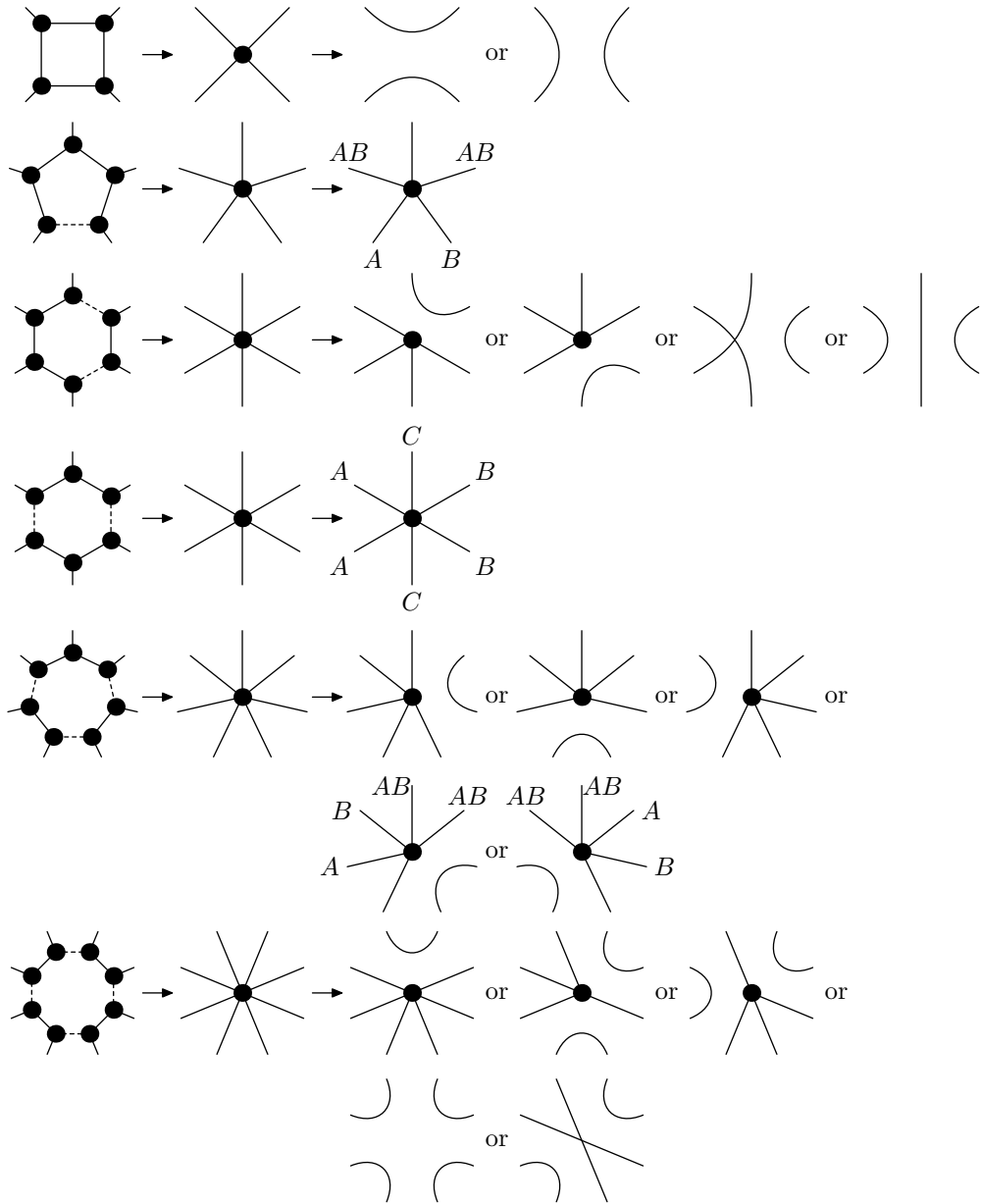

\begin{flushleft}
\epsfbox{scc2.55}
\vskip 2mm
\epsfbox{scc2.56}
\vskip 0mm
\epsfbox{scc2.57}
\vskip 2mm
\epsfbox{scc2.58}
\vskip 2mm
\epsfbox{scc2.59}
\vskip 2mm
\epsfbox{scc2.60}
\end{flushleft}
\caption{Modifications of the graph $H$ performed in the proof of
         Lemma~\ref{fig-rainbow-mindegree}. The edges of weight one are solid and
	 the edges of weight zero are dashed. The sets $A_w$, $B_w$ and $C_w$
	 are indicated by letters near the edges. Vertices of degree two obtained
	 through splittings are not depicted and some symmetric cases are omitted
	 in the case of a circuit of length eight.}
\label{fig-rainbow-mindegree}
\end{figure}

\begin{itemize}
\item if $k=4$, split the pair $w_1$ and $w_2$ or the pair
      $w_2$ and $w_3$ from $w$ in such a way that the resulting
      graph is $5$-odd-connected (at least one
      of the two splittings works by Lemma~\ref{lm-split4}).
\item if $k=5$ and the weight of the edge $v_4v_5$ is zero,
      set $A_w=\{e_1,e_3,e_5\}$ and $B_w=\{e_1,e_3,e_4\}$.
\item if $k=6$ and the weights of the edges $v_3v_4$ and $v_5v_6$ are zero,
      split the pair $w_3$ and $w_4$, $w_4$ and $w_5$, or      
      $w_5$ and $w_6$ from $w$ without creating edge-cuts of size
      one or three (one of the splitting works by Lemma~\ref{lm-split6}).
      If the pair $w_4$ and $w_5$ is split off, split further
      the pair $w_2$ and $w_6$, or the pair $w_3$ and $w_6$ from $w$
      again without creating edge-cuts of size one or three (one
      of the splitting works by Lemma~\ref{lm-split4}).
\item if $k=6$ and the weights of the edges $v_2v_3$ and $v_5v_6$ are zero,
      set $A_w=\{e_2,e_3\}$, $B_w=\{e_5,e_6\}$ and $C_w=\{e_1,e_4\}$.
\item if $k=7$ and the weights of the edges $v_2v_3$, $v_4v_5$ and $v_6v_7$ are zero,
      split one of the pairs $w_i$ and $w_{i+1}$ from $w$
      for $i\in\{2,3,4,5,6\}$ (the existence of such a splitting
      is guaranteed by Lemma~\ref{lm-split8}).
      If $w_3$ and $w_4$ is split off,
      set $A_w=\{e_1,e_2,e_6\}$ and $B_w=\{e_1,e_2,e_7\}$.
      If $w_5$ and $w_6$ is split off,
      set $A_w=\{e_1,e_2,e_7\}$ and $B_w=\{e_1,e_3,e_7\}$.
\item if $k=8$ and the weights of the edges $v_1v_2$, $v_3v_4$, $v_5v_6$ and
      $v_7v_8$ are equal to zero,
      split one of the pairs $w_i$ and $w_{i+1}$ from $w$
      for some $i\in\{1,\ldots,8\}$ (indices taken modulo eight)
      without creating edge-cuts of size one or three.
      This is possible by Lemma~\ref{lm-split8}. If $i$ is odd,
      then there are no further modifications to be performed. If $i$
      is even, one of the pairs $w_{i+3}$ and $w_{i+4}$,
      $w_{i+4}$ and $w_{i+5}$, and $w_{i+5}$ and $w_{i+6}$ is
      further split off from the vertex $w$ in such a way that
      the graph stays $5$-odd-connected (one of the splittings
      has this property by Lemma~\ref{lm-split6}). In case
      that the vertices $w_{i+4}$ and $w_{i+5}$ are split off,
      split further the pair of vertices $w_{i+2}$ and $w_{i+3}$ or
      the pair of vertices $w_{i+3}$ and $w_{i+6}$, again,
      keeping the graph $5$-odd-connected (and do not
      split off other pairs of vertices in the other cases).
      Lemma~\ref{lm-split4} guarantees that one of the two splittings
      works.
\end{itemize}

Fix a nowhere-zero $\ZZ_2^2$-flow $\varphi$ with the properties
described in Lemma~\ref{lm-flow} with respect to the sets $A_w$, $B_w$ and
$C_w$ as defined before (and where the sets $A_w$, $B_w$ and
$C_w$ are undefined, choose them arbitrarily).
The edges of $\varphi^{-1}(01)$ are colored with red,
the edges of $\varphi^{-1}(10)$ with green and
the edges of $\varphi^{-1}(11)$ with blue.
This defines the coloring of the edges of $\GG$ not contained in $F$.

Clearly, $F$ is a rainbow $2$-factor.
It remains to verify that the patterns of circuits with four
edges of weight one are as described in the statement of the lemma.
Let $C=v_1\ldots v_k$ be a circuit of $F$ consisting of four edges with weight one and
some edges with weight zero, and let $c_i$ be the color of the edge of $M$ incident with $v_i$.
We distinguish six cases based on the value of $k$ and the position of
zero-weight edges (symmetric cases are omitted):
\begin{itemize}
\item if $k=4$, then all the edges of $C$ have weight one. By the modification of $H$,
      it holds that $c_1=c_2$ or $c_2=c_3$. Hence,
      the pattern of $C$ is compatible with RRxx or xRRx.
\item if $k=5$ and the weight of $v_4v_5$ is zero,
      then either $c_1\not=c_3$, or $c_1=c_3\not\in\{c_4,c_5\}$.
      Since $C$ is incident with an odd number of edges
      of each color, its pattern is compatible with RxGxx or RRRGB.
\item if $k=6$ and the weights of $v_3v_4$ and $v_5v_6$ are zero,
      then $c_3=c_4$, or $c_4=c_5$ and $c_2=c_6$, or
      $c_4=c_5$ and $c_3=c_6$, or $c_5=c_6$.
      Hence, the pattern of $C$ is compatible with xxRRxx, xRxRRR or xRxGGR,
      xxRRRR or xxRGGR, or xxxxRR.
\item if $k=6$ and the weights of $v_2v_3$ and $v_5v_6$ are zero,
      then the pattern of $C$ contains all three possible colors or
      it is compatible with xRRxxx, xxxxRR or RxxGxx.
      In particular, it is not compatible with any of the patterns
      listed in the statement of the lemma.
\item if $k=7$ and the weights of $v_2v_3$, $v_4v_5$ and $v_6v_7$ are zero,
      then $c_i=c_{i+1}$ for some $i\in\{2,3,4,5,6\}$ by the modification of $H$.
      If $i$ is even, then the pattern of $C$ is compatible with xRRxxxx, xxxRRxx or
      xxxxxRR. If $i=3$, then $c_1\not=c_2$ or $c_1=c_2\not\in\{c_6,c_7\}$.
      Hence, the pattern of $C$ is compatible with RGRRxxx, RGBBxxx, RRGGxGB or
      RRGGxBG (unless $c_2=c_3$). Since $C$ is incident with an odd number
      of edges of each colors, its pattern is compatible with one of the patterns
      listed in the statement of the lemma. A symmetric argument applies
      if $i=5$ and either $c_1\not=c_7$ or $c_1=c_7\not\in\{c_2,c_3\}$.
\item if $k=8$ and the weights $v_iv_{i+1}$, $i=1,3,5,7$, then
      $c_i=c_{i+1}$ for $i\in\{1,\ldots,8\}$ by the modification of $H$.
      If there is such odd $i$, the pattern of $C$ is compatible with RRxxxxxx, xxRRxxxx,
      xxxxRRxx or xxxxxxRR. Otherwise, at least one of the following holds for
      some even $i$: $c_{i+3}=c_{i+4}$,
      $c_{i+4}=c_{i+5}$ or $c_{i+5}=c_{i+6}$. In the first and the last case,
      the pattern is again compatible with RRxxxxxx, xxRRxxxx, xxxxRRxx or xxxxxxRR.
      If $c_{i+4}=c_{i+5}$, then $c_{i+2}=c_{i+3}$ or $c_{i+3}=c_{i+6}$.
      Hence, the pattern of $C$ is compatible with xRRGGRRx, xRRGGBBx, xRRxRGGR,
      xRRxGRRG, xRRxGBBG or one of the patterns rotated by two, four or six
      positions. All these patterns are listed in the statement of the lemma.
\end{itemize}
\end{proof}

\section{Intermezzo}
\label{sec-intermezzo}

In order to help the reader to follow our arguments,
we present a proof that every bridgeless graph with $m$ edges
has a cycle cover of length at most $5m/3$ based on the Rainbow Lemma.
The proof differs both from the proof of Alon and Tarsi~\cite{bib-alon85+}
which is based on $6$-flows and the proof of
Bermond, Jackson and Jaeger~\cite{bib-bermond83+} based on $8$-flows;
on the other hand, its main idea resembles the proof of
Fan~\cite{bib-fan94} for cubic graphs.

Let $F$ be a set of disjoint circuits of a graph $\GG$,
e.g., $F$ can be a $2$-factor of a cubic graph $\GG$.
For a circuit $C$ of $F$ and for a set of edges of $E$ such that $C\cap E=\emptyset$,
we define $C(E)$ to be the set of vertices of $C$ incident
with an odd number of edges of $E$. If $C(E)$ has even cardinality,
it is possible to partition the edges of $C$ into two sets
$C(E)^A$ and $C(E)^B$ such that
\begin{itemize}
\item each vertex of $C(E)$ is incident with one edge of $C(E)^A$ and
      one edge of $C(E)^B$, and
\item each vertex of $C$ not contained in $C(E)$ is incident with either
      two edges of $C(E)^A$ or two edges of $C(E)^B$.
\end{itemize}
We will always assume that the number of edges of $C(E)^A$ does not exceed
the number of edges of $C(E)^B$, i.e., $|C(E)^A|\le|C(E)^B|$.
Note that if $C(E)=\emptyset$, then $C(E)^A$ contains no edges of $C$ and
$C(E)^B$ contains all the edges of $C$.

\begin{theorem}
\label{thm-intermezzo}
Let $\GG$ be a bridgeless graph with $m$ edges.
$\GG$ has a cycle cover of length at most $5m/3$.
\end{theorem}

\begin{proof}
If $\GG$ has a vertex $v$ of degree four or more, then, by Lemma~\ref{lm-split-cyclic},
$v$ has two neighbors $v_1$ and $v_2$ such that the graph $\GG.v_1vv_2$ is also bridgeless.
Let $\GG'$ be the graph $\GG.v_1vv_2$.
The number of edges of $\GG'$ is the same as the number of edges of $\GG$ and
every cycle of $\GG'$ corresponds to a cycle of $\GG$ of the same length.
Hence, a cycle cover of $\GG'$ corresponds to a cycle cover of $\GG$ of
the same length.
Through this process we can reduce any bridgeless graph to a bridgeless
graph with maximum degree three.
In particular, we can assume without loss of generality that
the graph $\GG$ has maximum degree three and $\GG$ is connected (if not,
consider each component separately).

If $\GG$ is a circuit, the statement is trivial.
Otherwise, we proceed as we now describe.
First, we suppress all vertices of degree two in $\GG$.
Let $\GG_0$ be the resulting
cubic (bridgeless) graph. We next assign each edge $e$ of $\GG_0$
the weight equal
to the number of edges in the path corresponding to $e$ in $\GG$.
In particular, the total weight of the edges of $\GG_0$ is equal to $m$.
Let $F_0$ be a rainbow $2$-factor with weight at least $2m/3$;
the existence of $F_0$ is guaranteed by Lemma~\ref{lm-rainbow-weighted}
applied with the weight function $-w$.

The $2$-factor $F_0$ corresponds to a set $F$ of disjoint circuits of
the graph $\GG$ which do not necessarily cover all the vertices of $\GG$.
Let $w_F$ be the weight of the edges contained in the $2$-factor $F_0$, and
$r$, $g$ and $b$ the weight of red, green and blue edges, respectively.
By symmetry, we can assume that $r\le g\le b$.
Since the weight $w_F$ of the edges contained in the $2$-factor $F_0$
is at least $2m/3$, the sum $r+g+b$ is at most $m/3$.
Finally, let $\RE$ be the set of edges of $\GG$ corresponding to red edges of $\GG_0$,
$\GE$ the set of edges corresponding to green edges, and
$\BE$ the set of edges corresponding to blue edges. By the choice of
edge-weights, the cardinality of $\RE$ is $r$,
the cardinality of $\GE$ is $g$ and the cardinality of $\BE$ is $b$.

The desired cycle cover of $\GG$ which is comprised
of three cycles can now be defined.
The first cycle $\CC_1$ consists of all the red and
green edges and the edges of $C(\RE\cup\GE)^A$
for all circuits $C$ of the $2$-factor $F$.
The second cycle $\CC_2$ consists of all the red and
green edges and the edges of $C(\RE\cup\GE)^B$
for all circuits $C$ of $F$.
Finally, the third cycle $\CC_3$ consists of all the red and
blue edges and the edges of $C(\RE\cup\BE)^A$
for all circuits $C$ of $F$.

Let us first verify that the cycles $\CC_1$, $\CC_2$ and $\CC_3$ cover
the edges of $G$. Clearly, every edge not contained in $F$, i.e.,
a red, green or blue edge, is covered by at least one of the cycles.
On the other hand,
every edge of $F$ is contained either in the cycle $\CC_1$ or the cycle $\CC_2$.
Hence, the cycles $\CC_1$, $\CC_2$ and $\CC_3$ form a cycle cover of $G$.

It remains to estimate the lengths of the cycles $\CC_1$, $\CC_2$ and $\CC_3$.
Each edge of $F$ is covered once by the cycles $\CC_1$ and $\CC_2$;
since $|C(E)^A|\le|C(E)^B|$ for every circuit $C$ of $F$, at most half
of the edges of $F$ is also covered by the cycle $\CC_3$. We conclude
that the total length of the constructed cycle cover is at most:
$$3r+2g+b+|F|+|F|/2\le 2(r+g+b)+3w_F/2=$$
$$3(r+g+b+w_F)/2+(r+g+b)/2\le 3m/2+m/6=5m/3\;\mbox{.}$$
This finishes the proof of the theorem.
\end{proof}

\section{Cubic graphs}
\label{sec-cubic}

We are now ready to improve the bound for cubic bridgeless graphs.

\begin{theorem}
\label{thm-cubic}
Every cubic bridgeless graph $\GG$ with $m$ edges
has a cycle cover comprised of three cycles of total length at most $34m/21$.
\end{theorem}

\begin{proof}
We present two bounds on the length of a cycle cover of $\GG$ and
the bound claimed in the statement of the theorem is eventually
obtained as their combination.
In both bounds,
the constructed cycle cover will consist of three cycles.
Fix a rainbow $2$-factor $F$ and
an edge-coloring of the edges not contained in $F$
with the red, green and blue colors
as described in Lemma~\ref{lm-rainbow-cubic}.
Let $\RE$, $\GE$ and $\BE$ be the sets of the red,
green and blue edges, respectively, and
let $r$, $g$ and $b$ be their numbers.
Finally, let $d_{\ell}$ be the number
of circuits of lengths $\ell$ contained in $F$.
By Lemma~\ref{lm-rainbow-cubic}, $d_3=0$.

{\bf The first cycle cover.}
Before we proceed with constructing the first cycle cover,
recall the notation of $C(E)^A$ and $C(E)^B$ introduced
before Theorem~\ref{thm-intermezzo}.
The cycle cover is comprised of three
cycles $\CC_1$, $\CC_2$ and $\CC_3$ which we next define.

Let $C$ be a circuit of the $2$-factor $F$. Lemma~\ref{lm-rainbow-cubic}
allows us to assume that if the length of $C$ is four, then
the pattern of $C$ is either BBBB or GGBB (otherwise, we can---for the purpose
of extending the cycles $\CC_1$, $\CC_2$ and $\CC_3$ to $C$---switch the roles
of the red, green and blue colors and the roles of the cycles $\CC_1$,
$\CC_2$ and $\CC_3$ in the remaining analysis; note that we are not
recoloring the edges, just apply the arguments presented in the next
paragraphs with respect to a different permutation of colors).
Similarly, we can assume that the pattern of
the circuit $C$ of length eight is one of the following $16$ patterns:
\begin{center}
BBBBBBBB, BBBBBBRR, BBBBRRRR, BBBBRRGG, \\
BBRRBBRR, BBRRBBGG, BBBBRBBR, BBBBRGGR, \\
BBRRBRRB, BBRRBGGB, RRGGBRRB, BBBBRBRB, \\
BBBRGRGB, RRBRBRBB, BBRBGBGR and BBRRGRGR.
\end{center}
For every circuit $C$ of $F$, we define subsets $C^1$, $C^2$ and $C^3$
of its edges. The subset $C^1$ is just $C(\RE\cup \GE)^A$.
The subset $C^2$ is either $C(\RE\cup \BE)^A$ or
$C(\RE\cup \BE)^B$---we choose the set with smaller intersection
with $C(\RE\cup \GE)^A$. Finally, $C^3$ is chosen so
that every edge of $C$ is contained in an odd number of the sets $C^i$;
explicitly, $C^3=C^1 \bigtriangleup C^2 \bigtriangleup C$.
Note that $C^3$ is either $C(\GE\cup \BE)^A$ or $C(\GE\cup \BE)^B$.

The cycle $\CC_1$ consists of all red and green edges and
the edges of $C^1$ for every circuit $C$ of $F$.
The cycle $\CC_2$ consists of all red and blue edges and
the edges of $C^2$ for every circuit $C$ of $F$.
Finally, the cycle $\CC_3$ consists of all green and blue edges and
the edges of $C^3$ for every circuit $C$ of $F$.
Clearly, the sets $\CC_1$, $\CC_2$ and $\CC_3$ form cycles.

We now estimate the number of the edges of $C$ contained in $\CC_1$,
$\CC_2$ and $\CC_3$. Let $\ell$ be the length of $C$.
The sum of the numbers edges contained in each of the cycles is:
\begin{eqnarray*}
|C^1| & + & |C^2|+|C^1 \bigtriangleup C^2 \bigtriangleup C| \\
 & = & |C^1\cup C^2|+|C^1\cap C^2|+|C\setminus (C^1\cup C^2)|+|C^1\cap C^2|\\
 & = & |C|+2|C^1\cap C^2|=\ell+2|C^1\cap C^2|
\end{eqnarray*}
Since $|C^1|=|C(\RE\cup \GE)^A|\le |C(\RE\cup \GE)^B|$, the number of edges of
$C^1$ is at most $\ell/2$.
By the choice of $C^2$, $|C^1\cap C^2|\le |C^1|/2\le\ell/4$.
Hence, the sets $\CC_1$, $\CC_2$ and $\CC_3$
contain at most $\ell+2\lfloor\ell/4\rfloor$ edges of the circuit $C$.

The estimate on the number of edges of $C$ contained in the cycles
$\CC_1$, $\CC_2$ and $\CC_3$ can further be improved if the length
of the circuit $C$ is four: if the pattern of $C$ is BBBB,
then $C^1=C(\RE\cup \GE)^A=\emptyset$ and
thus $C^1\cap C^2=\emptyset$.
If the pattern is GGBB,
then $C^1\cap C^2=C(\RE\cup \GE)^A\cap C(\RE\cup \BE)^A=\emptyset$.
In both the cases, it holds that $C^1\cap C^2=\emptyset$ and
thus the cycles $\CC_1$, $\CC_2$ and $\CC_3$ contain
(at most) $|C|+2|C^1\cap C^2|=|C|=\ell=4$ edges of $C$.

\begin{figure}
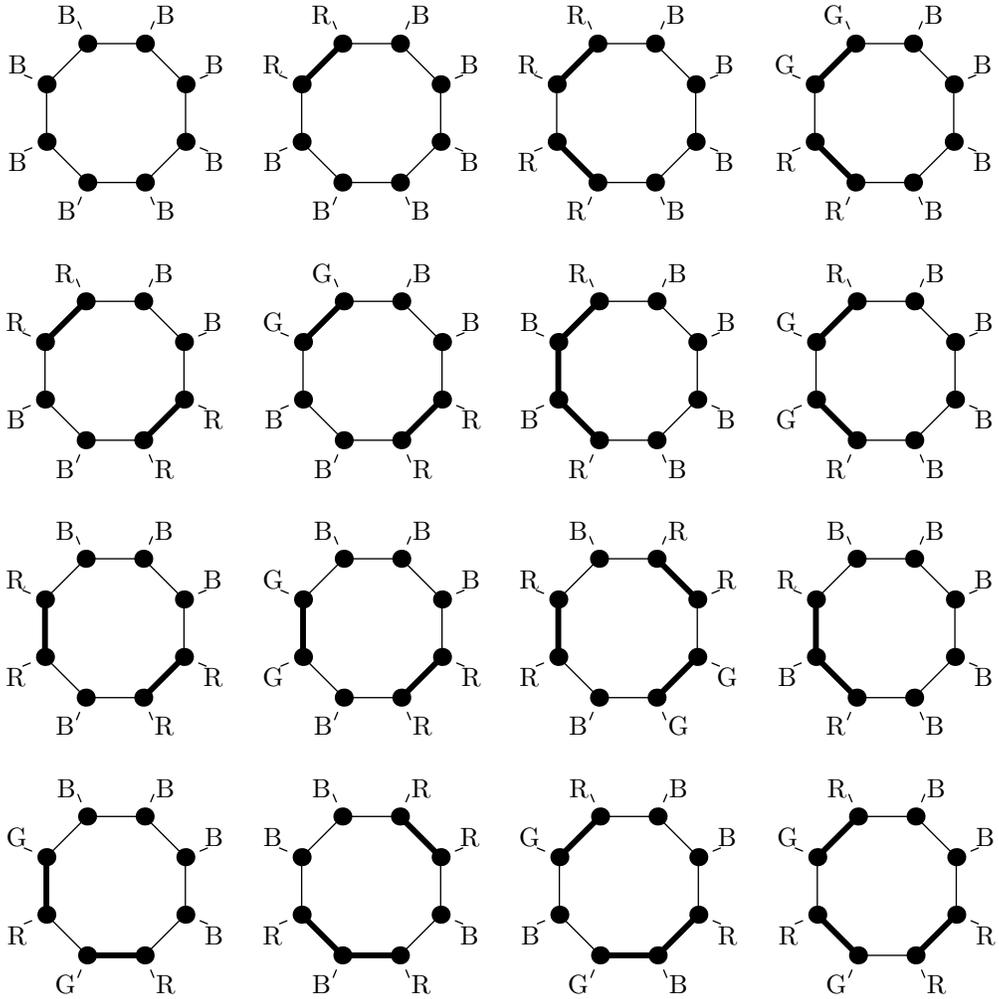

\begin{center}
\epsfbox{scc.3}\hskip 5mm
\epsfbox{scc.4}\hskip 5mm
\epsfbox{scc.5}\hskip 5mm
\epsfbox{scc.6}\vskip 5mm
\epsfbox{scc.7}\hskip 5mm
\epsfbox{scc.8}\hskip 5mm
\epsfbox{scc.9}\hskip 5mm
\epsfbox{scc.10}\vskip 5mm
\epsfbox{scc.11}\hskip 5mm
\epsfbox{scc.12}\hskip 5mm
\epsfbox{scc.13}\hskip 5mm
\epsfbox{scc.14}\vskip 5mm
\epsfbox{scc.15}\hskip 5mm
\epsfbox{scc.16}\hskip 5mm
\epsfbox{scc.17}\hskip 5mm
\epsfbox{scc.18}
\end{center}
\caption{The sets $C^1=C(\RE\cup \GE)^A$ for circuits $C$ with length eight;
         the edges contained in the set are drawn bold.}
\label{fig-cubic-1}
\end{figure}

Similarly, the estimate on the number of edges of $C$ contained
in the cycles can be improved if the length of $C$ is eight. 
As indicated in Figure~\ref{fig-cubic-1},
it holds that $|C^1|=|C(\RE\cup \GE)^A|\le 3$. Hence,
$|C^1\cap C^2|\le|C^1|/2\le 3/2$.
Consequently, $|C^1\cap C^2|\le 1$ and
the number of edges of $C$ included
in the cycles $\CC_1$, $\CC_2$ and $\CC_3$
is at most $|C|+2|C^1\cap C^2|\le 8+2=10$.

Based on the analysis above, we can conclude that the cycles $\CC_1$,
$\CC_2$ and $\CC_3$ contain at most the following number of edges of the $2$-factor $F$
in total:
\begin{eqnarray}
& & 2d_2+4d_4+7d_5+8d_6+9d_7+10d_8+13d_9+14d_{10}+15d_{11}+
     \sum_{\ell=12}^\infty\frac{3\ell}{2}d_{\ell}\nonumber \\
& = & \frac{3}{2}\sum_{\ell=2}^\infty\ell d_{\ell}
 -d_2-2d_4-\frac{1}{2}d_5-d_6-\frac{3}{2}d_7-2d_8-\frac{1}{2}d_9-d_{10}-\frac{3}{2}d_{11}\;\mbox{.} \label{eq-cubic-1}
\end{eqnarray}
Since the $2$-factor $F$ contains $2m/3$ edges, the estimate (\ref{eq-cubic-1})
translates to:
\begin{equation}
m-d_2-2d_4-\frac{1}{2}d_5-d_6-\frac{3}{2}d_7-2d_8-\frac{1}{2}d_9-d_{10}-\frac{3}{2}d_{11}\;\mbox{.}
  \label{eq-cubic-2}
\end{equation}
Since each red, green or blue edge is contained in exactly two of the cycles
$\CC_1$, $\CC_2$ and $\CC_3$ and there are $m/3$ such edges,
the total length of the cycle cover of $\GG$
formed by $\CC_1$, $\CC_2$ and $\CC_3$ does not exceed:
\begin{equation}
\frac{5m}{3}-d_2-2d_4-\frac{1}{2}d_5-d_6-\frac{3}{2}d_7-2d_8-\frac{1}{2}d_9-d_{10}-\frac{3}{2}d_{11}\;\mbox{.}
  \label{eq-cubic-SC}
\end{equation}
This finishes the construction and
the analysis of the first cycle cover of $\GG$.

{\bf The second cycle cover.}
We reuse the $2$-factor $F$ and the coloring of the edges of $\GG$
by red, green and blue colors from the construction of the first cycle cover.
As long as the graph $H=\GG/F$
contains a red circuit, choose a red circuit of $H=\GG/F$ and recolor its edges
with blue. Similarly, recolor edges of green circuits with blue. The modified
edge-coloring still gives a rainbow $2$-factor.
Let $\RE'$, $\GE'$ and $\BE'$ be the sets of red, green and
blue in the modified edge-coloring and $r'$, $g'$ and $b'$ their cardinalities.

The construction of the cycle cover now follows
the lines of the proof
of Theorem~\ref{thm-intermezzo}. The first cycle $\CC_1$ is formed
by the red and green edges and the edges of $C(\RE'\cup \GE')^A$ for every
circuit $C$ of the $2$-factor $F$. The cycle $\CC_2$ is also formed
by the red and green edges and, in addition, the edges of $C(\RE'\cup \GE')^B$
for every circuit $C$ of $F$. Finally, the cycle $\CC_3$ is formed
by the red and blue edges and the edges of $C(\RE'\cup \BE')^A$ for every
circuit $C$ of $F$. Clearly, the sets $\CC_1$, $\CC_2$ and $\CC_3$
are cycles of $\GG$ and they cover all the edges of $\GG$.

Let us now estimate the lengths of the cycles $\CC_1$, $\CC_2$ and $\CC_3$.
Each red edge is contained in all the three cycles, each green edge
in two cycles and each blue edge in one cycle. Each edge of a circuit
$C$ of length $\ell$ of the $2$-factor $F$ is contained either in $\CC_1$ or
in $\CC_2$ and at most half of the edges of $C$ is also contained
in the cycle $\CC_3$. Hence, the total length of the cycles
$\CC_1$, $\CC_2$ and $\CC_3$ is at most:
\begin{equation}
3r'+2g'+b'+\sum_{\ell=2}^\infty\left\lfloor\frac{3\ell}{2}\right\rfloor d_{\ell}\;\mbox{.}  
  \label{eq-cubic-3}
\end{equation}
Since the red edges form an acyclic subgraph of $\GG/F$, the number of red
edges is at most the number of the cycles of $F$, i.e., 
$d_2+d_3+d_4+d_5+\ldots$ Similarly, the number
of green edges does not exceed the number of the cycles of $F$.
Since $r'+g'+b'=m/3$, the expression (\ref{eq-cubic-3}) can be estimated
from above by
\begin{equation}
\frac{m}{3}+2r'+g'+\sum_{\ell=2}^\infty\left\lfloor\frac{3\ell}{2}\right\rfloor d_{\ell}\le
\frac{m}{3}+\sum_{\ell=2}^\infty 3 d_{\ell}+
  \sum_{\ell=2}^\infty\left\lfloor\frac{3\ell}{2}\right\rfloor d_{\ell}=
\frac{m}{3}+\sum_{\ell=2}^\infty\left\lfloor\frac{3\ell}{2}+3\right\rfloor d_{\ell}\;\mbox{.}
  \label{eq-cubic-4}
\end{equation}
Since the $2$-factor $F$ contains $2m/3=2d_2+3d_3+4d_4+\ldots$ edges,
the bound (\ref{eq-cubic-4}) on the number of edges contained
in the constructed cycle cover can be rewritten to
\begin{equation}
\frac{m}{3}+\frac{7\cdot 2\cdot m}{4\cdot 3}+\sum_{\ell=2}^\infty
  \left(\left\lfloor\frac{3\ell}{2}+3\right\rfloor-\frac{7\ell}{4}\right) d_{\ell} \le
\frac{3m}{2}+\sum_{\ell=2}^{10}
  \left(\left\lfloor\frac{3\ell}{2}+3\right\rfloor-\frac{7\ell}{4}\right) d_{\ell}\;\mbox{.}
  \label{eq-cubic-5}
\end{equation}
Note that the last inequality follows
from the fact that
$\left\lfloor\frac{3\ell}{2}+3\right\rfloor-\frac{7\ell}{4}\le
  \frac{3\ell}{2}-\frac{7\ell}{4}+3=3-\frac{\ell}{4}\le 0$ for $\ell\ge 12$ and
the expression $\left\lfloor\frac{3\ell}{2}+3\right\rfloor-\frac{7\ell}{4}=-1/4$
is also non-positive for $\ell=11$.
The estimate (\ref{eq-cubic-5}) can be expanded to the following form (recall that $d_3=0$):
\begin{equation}
\frac{3m}{2}+\frac{5}{2}d_2+2d_4+\frac{5}{4}d_5+\frac{3}{2}d_6+\frac{3}{4}d_7+d_8+\frac{1}{4}d_9+\frac{1}{2}d_{10}\;\mbox{.}
  \label{eq-cubic-LC}
\end{equation}
We remark that the bound (\ref{eq-cubic-LC}) could also be obtained
by a suitable substitution to the bounds presented in~\cite{bib-fan94}.
However, we decided to present the construction to make the paper
self-contained.

The length of the shortest cycle cover of $\GG$ with three cycles
exceeds neither the bound
given in (\ref{eq-cubic-SC}) nor the bound given in (\ref{eq-cubic-LC}).
Hence, the length of such cycle cover of $\GG$ is bounded by any
convex combination of the two bounds, in particular, by the following:
$$\frac{5}{7}\cdot\left(
    \frac{5m}{3}-d_2-2d_4-\frac{1}{2}d_5-d_6-\frac{3}{2}d_7-2d_8-\frac{1}{2}d_9-d_{10}-\frac{3}{2}d_{11}
           \right) +$$
$$\frac{2}{7}\cdot\left(
    \frac{3m}{2}+\frac{5}{2}d_2+2d_4+\frac{5}{4}d_5+\frac{3}{2}d_6+\frac{3}{4}d_7+d_8+\frac{1}{4}d_9+\frac{1}{2}d_{10}
           \right)=$$
$$\frac{34m}{21}-\frac{6}{7}d_4-\frac{2}{7}d_6-\frac{6}{7}d_7-\frac{8}{7}d_8-\frac{2}{7}d_9-\frac{4}{7}d_{10}-\frac{15}{14}d_{11}\le
  \frac{34m}{21}\;\mbox{.}$$
The proof of Theorem~\ref{thm-cubic} is now completed.
\end{proof}

\section{Graphs with minimum degree three}
\label{sec-mindegree}

We first show that it is enough to prove the main theorem of this section (Theorem~\ref{thm-main})
for graphs that do not contain parallel edges of certain type.
We state and prove four auxiliary lemmas.
The first two lemmas deal with the cases when there is a vertex
incident only with parallel edges leading to the same vertex.

\begin{lemma}
\label{lm-parallel-leaf}
Let $\GG$ be an $m$-edge bridgeless graph with vertices $v_1$ and $v_2$
joined by $k\ge 3$ parallel edges. If the degree of $v_1$ is $k$,
the degree of $v_2$ is at least $k+3$ and the graph $\GG'=\GG\setminus v_1$
has a cycle cover with three cycles of length at most $44(m-k)/27$,
then $\GG$ has a cycle cover with three cycles of length at most $44m/27$.
\end{lemma}

\begin{proof}
Let $\CC_1$, $\CC_2$ and $\CC_3$ be the cycles of total length
at most $44(m-k)/27$ covering the edges of $\GG'$ and $e_1,\ldots,e_k$
the $k$ parallel edges between the vertices $v_1$ and $v_2$.
If $k$ is even, add the edges $e_1,\ldots,e_k$ to $\CC_1$.
If $k$ is odd, add the edges $e_1,\ldots,e_{k-1}$ to $\CC_1$ and
the edges $e_{k-1}$ and $e_k$ to $\CC_2$. Clearly, we have obtained
a cycle cover of $\GG$ with three cycles. The length of the cycles
is increased at most by $k+1$ and thus it is at most
$$\frac{44m-44k}{27}+k+1=\frac{44m-17k+27}{27}\le\frac{44m}{27}\;\mbox{.}$$
\end{proof}

\begin{lemma}
\label{lm-parallel-leaf-special}
Let $\GG$ be an $m$-edge bridgeless graph with vertices $v_1$ and $v_2$
joined by $k\ge 4$ parallel edges. If the degree of $v_1$ is $k$,
the degree of $v_2$ is $k+2$ and the graph $\GG'$ obtained from $\GG$
by removing all the edges between $v_1$ and $v_2$ and suppressing
the vertex $v_2$
has a cycle cover with three cycles of length at most $44(m-k-1)/27$,
then $\GG$ has a cycle cover with three cycles of length at most $44m/27$.
\end{lemma}

\begin{proof}
Let $\CC_1$, $\CC_2$ and $\CC_3$ be the cycles of total length
at most $44(m-k-1)/27$ covering the edges of $\GG'$ and $e_1,\ldots,e_k$
the $k$ parallel edges between the vertices $v_1$ and $v_2$.
Let $v'$ and $v''$ be the two neighbors of $v_2$ distinct from $v_1$.
Note that it can hold that $v'=v''$.
The edges $v_2v'$ and $v_2v''$ are included in those cycles $\CC_i$ that
contain the edge $v'v''$. The edges $e_1,\ldots,e_{k-1}$ are included
to $\CC_1$. In addition, the edge $e_k$ is included to $\CC_1$
if $k$ is even. If $k$ is odd, the edges $e_{k-1}$ and $e_k$ are included to $\CC_2$.

The length of all the cycles is increased by at most $3+k+1=k+4$.
Hence, the total length of the cycle cover is at most
$$\frac{44m-44k-44}{27}+k+4=\frac{44m-17k+64}{27}\le\frac{44m}{27}\;\mbox{.}$$
\end{proof}

In the next two lemmas, we deal with the case that each of the two vertices
joined by parallel edges is adjacent to another vertex.

\begin{lemma}
\label{lm-parallel-inner}
Let $\GG$ be an $m$-edge bridgeless graph with vertices $v_1$ and $v_2$
joined by $k\ge 2$ parallel edges. If the degree of $v_1$ is at least $k+1$,
the degree of $v_2$ is at least $k+2$ and the graph $\GG'$ obtained
by contracting
all the edges between $v_1$ and $v_2$ has a cycle cover with three cycles
of length at most $44(m-k)/27$,
then $\GG$ has a cycle cover with three cycles of length at most $44m/27$.
\end{lemma}

\begin{proof}
Let $\CC_1$, $\CC_2$ and $\CC_3$ be the cycles of total length
at most $44(m-k)/27$ covering the edges of $\GG'$ and $e_1,\ldots,e_k$
the $k$ parallel edges between the vertices $v_1$ and $v_2$.
By symmetry, we can assume that the cycles
$\CC_1,\ldots,\CC_{i_0}$ contain an odd number of edges
incident with $v_1$ and the cycles $\CC_{i_0+1},\ldots,\CC_3$
contain an even number of such edges for some $i_0\in\{0,1,2,3\}$.
Since $\CC_1,\ldots,\CC_3$ form a cycle cover of $\GG'$,
if $v_1$ is incident with an odd number of edges of $\CC_i$, $i=1,2,3$,
then $v_2$ is incident with an odd number of edges of $\CC_i$ and
vice versa. 

The edges are added to the cycles $\CC_1$, $\CC_2$ and $\CC_3$ as follows
based on the value of $i_0$ and the parity of $k$:
\begin{center}
\begin{tabular}{|c|c|c|c|c|}
\hline
$i_0$ & $k$ & $\CC_1$ & $\CC_2$ & $\CC_3$ \\
\hline
$0$ & odd  & $e_1,\ldots,e_{k-1}$ & $e_{k-1},e_k$ & \\
$0$ & even & $e_1,\ldots,e_k$ & & \\
$1$ & odd  & $e_1,\ldots,e_k$ & & \\
$1$ & even & $e_1,\ldots,e_{k-1}$ & $e_{k-1},e_k$ & \\
$2$ & odd  & $e_1,\ldots,e_k$ & $e_k$ & \\
$2$ & even & $e_1,\ldots,e_{k-1}$ & $e_k$ & \\
$3$ & odd  & $e_1,\ldots,e_{k-2}$ & $e_{k-1}$ & $e_k$ \\
$3$ & even & $e_1,\ldots,e_{k-1}$ & $e_{k-1}$ & $e_k$ \\
\hline
\end{tabular}
\end{center}
Clearly, we have obtained a cycle cover of $\GG$ with three cycles.
The length of the cycles is increased at most by $k+1$ and thus it is at most
$$\frac{44m-44k}{27}+k+1=\frac{44m-17k+27}{27}\le\frac{44m}{27}\;\mbox{.}$$
\end{proof}

\begin{lemma}
\label{lm-parallel-suppress}
Let $\GG$ be an $m$-edge bridgeless graph with vertices $v_1$ and $v_2$
joined by $k\ge 3$ parallel edges. If the degrees of $v_1$ and $v_2$
are $k+1$ and the graph $\GG'$ obtained by contracting all the edges
between $v_1$ and $v_2$ and suppressing the resulting vertex of degree two
has a cycle cover with three cycles of length at most $44(m-k-1)/27$,
then $\GG$ has a cycle cover with three cycles of length at most $44m/27$.
\end{lemma}

\begin{proof}
Let $\CC_1$, $\CC_2$ and $\CC_3$ be the cycles of total length
at most $44(m-k-1)/27$ covering the edges of $\GG'$, let $e_1,\ldots,e_k$
be the $k$ parallel edges between the vertices $v_1$ and $v_2$, and
let $v'_i$ be the other neighbor of $v_i$, $i=1,2$.
Add the edges incident with $v_1v'_1$ and $v_2v'_2$
to those cycles $\CC_1$, $\CC_2$ and $\CC_3$ that contain the edge
$v'_1v'_2$ and then proceed as in the proof of Lemma~\ref{lm-parallel-inner}.
The length of the cycles is increased by at most $3+k+1=k+4$ and
thus it is at most
$$\frac{44m-44k-44}{27}+k+4=\frac{44m-17k+64}{27}\le\frac{44m}{27}$$
where the last inequality holds
unless $k=3$. If $k=3$ and the edge $v'_1v'_2$ is contained
in at most two of the cycles, the length is increased by at most $2+k+1=k+3=6$.
If $k=3$ and the edge $v'_1v'_2$ is contained in three
of the cycles, each of the parallel edges is added to exactly one
of the cycles and thus the length is increased by at most $3+k=6$.
In both cases, the length of the new cycle cover can be estimated as follows:
$$\frac{44m-44\cdot 3-44}{27}+6=\frac{44m-14}{27}\le\frac{44m}{27}\;\mbox{.}$$
\end{proof}

We are now ready to prove our bound for graphs with minimum degree three.

\begin{theorem}
\label{thm-main}
Let $\GG$ be a bridgeless graph with $m$ edges and
with minimum degree three or more. The graph
$\GG$ has a cycle cover of total length at most $44m/27$
that is comprised of at most three cycles.
\end{theorem}

\begin{proof}
By Lemmas~\ref{lm-parallel-leaf}--\ref{lm-parallel-suppress},
we can assume without loss of generality that if vertices $v_1$ and $v_2$ of $\GG$
are joined by $k$ parallel edges,
then either $k=2$ and the degrees of both $v_1$ and $v_2$ are equal to $k+1=3$, or
$k=3$, the degree of $v_1$ is $k=3$ and the degree of $v_2$ is $k+2=5$ (in particular,
both $v_1$ and $v_2$ have odd degrees).
Note that the graphs $\GG'$ from the statement
of Lemmas~\ref{lm-parallel-leaf}--\ref{lm-parallel-suppress}
are also bridgeless graphs with minimum degree three and have fewer
edges than $\GG$ which implies that the reduction process described
in Lemmas~\ref{lm-parallel-leaf}--\ref{lm-parallel-suppress}
eventually finishes.

As the first step, we modify the graph $\GG$ into bridgeless graphs
$\GG_1,\GG_2,\ldots$ eventually obtaining a bridgeless graph $\GG'$
with vertices of degree two, three and four. Set $\GG_1=\GG$.
If $\GG_i$ has no vertices of degree five or more, let $\GG'=\GG_i$.
If $\GG_i$ has a vertex $v$ of degree five or more,
then Lemma~\ref{lm-split-cyclic} yields that
there are two neighbors, say $v_1$ and $v_2$, of $v$ such that
the graph $\GG_i.v_1vv_2$ is also bridgeless.
We set $\GG_{i+1}$ to be the graph $\GG_i.v_1vv_2$.
We continue while the graph $\GG_i$ has vertices of degree five or more.
Clearly, the final graph $\GG'$ has the same number of edges
as the graph $\GG$ and every cycle of $\GG'$ corresponds to a cycle of $\GG$.

Next, each edge of $\GG'$ is assigned weight one,
each vertex of degree four is expanded to two vertices of degree three
as described in Lemma~\ref{lm-expand} and
the edge between the two new vertices of degree three is assigned weight zero (note that
the vertex splitting preserves the parity of the degree of the split vertex and
thus no vertex of degree four is incident with parallel edges).
The resulting graph is denoted by $\GG_0$.
Note that every cycle $C$ of $\GG_0$ corresponds to a cycle $C'$ of $\GG$ and
the length of $C'$ in $\GG$ is equal to the sum of the weights of the edges of $C$.
Next, the vertices of degree two in $\GG_0$ are suppressed and each edge $e$
is assigned the weight equal to the sum of the weights of edges of the path
of $\GG_0$ corresponding to $e$. The resulting graph is denoted by $\GG'_0$.
Clearly, $\GG'_0$ is a cubic bridgeless graph.
Also note that all the edges of weight zero in $\GG_0$ are also contained
in $\GG'_0$, no vertex of $\GG'_0$ is incident with two edges of weight zero, and
the total weight of the edges of $\GG'_0$ is equal to $m$.

We apply Lemma~\ref{lm-rainbow-mindegree} to the cubic graph $\GG'_0$.
Let $F'_0$ be the rainbow $2$-factor of $\GG'_0$ and
let $F_0$ be the cycle of $\GG_0$ corresponding to the $2$-factor of $F'_0$.
Note that $F_0$ is a union of vertex-disjoint circuits.
Let $\RE_0$, $\GE_0$ and $\BE_0$ be the sets of edges of $\GG_0$
contained in paths corresponding to red, green and blue edges in $\GG'_0$.
Let $r_0$ be the weight of the red edges in $G_0$,
$g_0$ the weight of green edges and $b_0$ the weight of blue edges.
Lemma~\ref{lm-rainbow-mindegree} yields $r_0+g_0+b_0\ge m/3$.

We construct two different cycle covers, each comprised of three cycles, and
eventually combine the bounds on their lengths to obtain the bound claimed
in the statement of the theorem.

{\bf The first cycle cover.}
The first cycle cover that we construct
is a cycle cover of the graph $\GG_0$ (which yields a cycle cover
of $\GG$ of the same length as explained earlier).
Let $d_{\ell}$ be the number of circuits of $F_0$ of weight $\ell$.
Note that $d_3$ can be non-zero since a circuit of weight three need not
have length three in $\GG_0$.

Recall now the notation $C(E)^A$ and $C(E)^B$
used in the proof of Theorem~\ref{thm-intermezzo} for circuits $C$ and
set $E$ of edges that are incident with even number of vertices of $C$.
In addition,
$C(E)_*^A$ denotes the edges of $C(E)^A$ with weight one and
$C(E)_*^B$ denotes such edges of $C(E)^B$.
In the rest of the construction of the first cycle cover,
we always assume that $|C(E)_*^A|\le|C(E)_*^B|$.
The sets $\CC_1$, $\CC_2$ and $\CC_3$ are completed to cycles
in a way similar to that used in the proof of Theorem~\ref{thm-intermezzo}. 

As in the proof of Theorem~\ref{thm-cubic}, we define subsets $C^1$, $C^2$ and $C^3$
of the edges of a circuit $C$ of $F_0$.
The subset $C^1$ are just the edges $C(\RE_0\cup\GE_0)^A$.
The subset $C^2$ are either the edges of $C(\RE_0\cup \BE_0)^A$ or
$C(\RE_0\cup \BE_0)^B$---we choose the set with fewer edges
with weight one in common with $C^1=C(\RE_0\cup \GE_0)^A$.
Finally, $C^3=C^1 \bigtriangleup C^2 \bigtriangleup C$.

The cycle $\CC_1$ consists of all the red and green edges,
i.e., the edges contained in $\RE_0\cup\GE_0$, and
the edges of $C^1$ for every circuit $C$ of $F_0$.
The cycle $\CC_2$ consists of the red and blue edges and
the edges of $C^2$.
Finally, the cycle $\CC_3$ consists of the green and blue edges and
the edges of $C^3$.

We now estimate the number of the edges of $C$ of weight one
contained in $\CC_1$, $\CC_2$ and $\CC_3$.
Let $C_*$ be the edges of weight one contained in the circuit $C$,
$\ell=|C_*|$ and $C_*^i=C^i\cap C_*$ for $i=1,2,3$.
By the choice of $C^2$, the number of edges of weight one
in $C^1\cap C^2$ is $|C^1_*\cap C^2_*|\le |C^1_*|/2$.
Consequently, the number of edges of $C$ of weight one contained
in the cycles $\CC_1$, $\CC_2$ and $\CC_3$ is:
\begin{eqnarray*}
|C^1_*|+|C^2_*|+|C^1_* \bigtriangleup C^2_*\bigtriangleup C_*| & = & \\
|C^1_*\cup C^2_*|+|C^1_*\cap C^2_*|+
|C_*\setminus(C^1_*\cup C^2_*)|+|C^1_*\cap C^2_*| & = & \\
|C_*|+2|C^1_*\cap C^2_*|\;\mbox{.} & &
\end{eqnarray*}
Since $|C(\RE_0\cup\GE_0)_*^A|\le |C(\RE_0\cup\GE_0)_*^B|$,
the number of edges contained in the set
$C_*^1=C(\RE_0\cup\GE_0)_*^A$ is at most $\ell/2$.
By the choice of $C^2$, $|C_*^1\cap C_*^2|\le |C_*^1|/2$.
Consequently, it holds that
\begin{equation}
|C_*^1\cap C_*^2|\le |C_*^1|/2\le\ell/4
\label{eq-mindegree-1}
\end{equation}
and eventually conclude that the sets $\CC_1$, $\CC_2$ and $\CC_3$ contain
at most $\ell+2\lfloor\ell/4\rfloor$ edges of the circuit $C$ with weight one.

If $\ell=4$,
the estimate given in (\ref{eq-mindegree-1}) can be further refined.
Let $C'$ be the circuit of $G'$ corresponding to $C$.
Clearly, $C'$ is a circuit of length four.
Color a vertex $v$ of the circuit $C'$
\begin{description}
\item[red] if $v$ has degree three and is incident with a red edge, or $v$ has degree four and is incident with green and blue edges,
\item[green] if $v$ has degree three and is incident with a green edge, or $v$ has degree four and is incident with red and blue edges,
\item[blue] if $v$ has degree three and is incident with a blue edge, or $v$ has degree four and is incident with red and green edges, and
\item[white] otherwise.
\end{description}
Observe that either $C'$ contains a white vertex or
it contains an even number of red vertices,
an even number of green vertices and an even number of blue vertices.
If $C'$ contains a white vertex, it is easy to verify that
\begin{equation}
|C_*^1|=|C(\RE_0\cup\GE_0)^A_*|\le 1\label{eq-mindegree-pat}
\end{equation}
for a suitable permutation
of red, green and blue colors; the inequality (\ref{eq-mindegree-pat}) also holds
if $C'$ contains two adjacent vertices of the same color (see Figure~\ref{fig-mindegree-pat}).
Next, we show that every such circuit contains a white vertex or two adjacent vertices
with the same color (and thus (\ref{eq-mindegree-pat}) always holds).

\begin{figure}
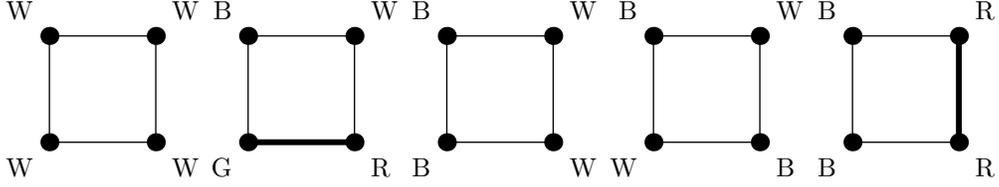

\begin{center}
\epsfbox{scc2.50}
\epsfbox{scc2.51}
\epsfbox{scc2.52}
\epsfbox{scc2.53}
\epsfbox{scc2.54}
\end{center}
\caption{An improvement for circuits of length four considered in the proof
         of Theorem~\ref{thm-main}. The letters R, G, B and W stand for red, green, blue and white colors. Note that it is possible to freely permute the red, green and blue colors.
	 The edges included to $C(\RE_0\cup\GE_0)^A$ are bold.
	 Symmetric cases are omitted.}
\label{fig-mindegree-pat}
\end{figure}

If the circuit of $F'_0$ corresponding to $C$ in $\GG'_0$
contains an edge of weight two or more, then $C$ contains
a white vertex and the estimate (\ref{eq-mindegree-pat}) holds.
Otherwise, all vertices of $C$ have degree three in $\GG_0$ and
thus the circuit $C$ is also contained in $F'_0$.
Since the edges of $C$ have weight zero and one only,
the pattern of $C$ is one of the patterns listed
in Lemma~\ref{lm-rainbow-mindegree}.
A close inspection of possible patterns of $C'$
yields that the cycle $C'$ contains a white vertex or
it contains two adjacent vertices with the same color.
We conclude that the estimate (\ref{eq-mindegree-pat}) applies.
Hence, if $\ell=4$, the estimate (\ref{eq-mindegree-1}) can be improved to $0$.

We now estimate the length of the cycle cover of $\GG_0$ formed by the cycles $\CC_1$, $\CC_2$ and $\CC_3$.
Since each red, green and blue edge is covered by exactly two of the cycles, we conclude that:
\begin{eqnarray}
2(r_0+g_0+b_0)+2d_2+3d_3+4d_4+7d_5+8d_6+9d_7+\sum_{\ell=8}^\infty\frac{3\ell}{2}d_{\ell} & = & \nonumber \\
2(r_0+g_0+b_0)+\frac{3}{2}\sum_{\ell=2}^{\infty}\ell d_{\ell}-d_2-3d_3/2-2d_4-d_5/2-d_6-3d_7/2 & = & \nonumber \\
\frac{3m}{2}+\frac{r_0+g_0+b_0}{2}-d_2-3d_3/2-2d_4-d_5/2-d_6-3d_7/2\;\mbox{.} \label{eq-mindegree-SC}
\end{eqnarray}
Note that we have used the fact that the sum $r_0+g_0+b_0+\sum_{\ell=2}^\infty \ell d_{\ell}$
is equal to the number of the edges of $\GG$.

{\bf The second cycle cover.}
The second cycle cover is constructed in an auxiliary graph $\GG''$
which we now describe. Every vertex $v$ of $\GG$ is eventually
split to a vertex of degree three or four in $\GG'$. 
The vertex of degree four is then expanded. Let $r(v)$ be
the vertex of degree three obtained from $v$ or one of the two vertices obtained
by the expansion of the vertex of degree four obtained from $v$.
By the construction of $F_0$, each $r(v)$ is contained in a circuit
of $F_0$. The graph $\GG''$ is constructed from the graph $\GG_0$ as follows:
every vertex of $\GG_0$ of degree two not contained in $F_0$
that is obtained by splitting from a vertex $v$ is identified
with the vertex $r(v)$. The edges of weight zero
contained in a cycle of $F_0$ are then contracted.
Let $F$ be the cycle of $\GG''$
corresponding to the cycle $F_0$ of $\GG_0$. Note that $F$ is formed by disjoint
circuits and it contains $d_{\ell}$ circuits of weight/length $\ell$.

Observe that $\GG''$ can be obtained from $\GG$
by splitting some of its vertices (perform exactly those splittings
yielding vertices of degree two contained in the circuits of $F_0$) and
then expanding some vertices.
In particular, every cycle of $\GG''$ is also
a cycle of $\GG$. Edges of weight one of $\GG''$ one-to-one correspond
to edges of weight one of $\GG_0$, and edges of weight zero of $\GG''$
correspond to edges of weight zero of $\GG_0$ not contained in $F_0$.
Hence, the weight of a cycle in $\GG''$ is
the length of the corresponding cycle in $\GG$.

The edges not contained in $F$ are red, green and blue (as in $\GG_0$).
Each circuit of $F$ is incident either with an odd number of red edges,
an odd number of green edges and an odd number of blue edges, or
with an even number of red edges, an even number of green edges and
an even number of blue edges (chords are counted twice). 
Let $H=\GG''/F$. If $H$ contains a red 
circuit (which can be a loop), recolor such a circuit to blue.
Similarly, recolor green circuits to blue. Let $\RE$, $\GE$ and $\BE$
be the resulting sets of red, green and blue edges and $r$, $g$ and $b$
their weights. Clearly, $r+g+b=r_0+g_0+b_0$. Also note that each
circuit of $F$ is still incident either with an odd number of red edges,
an odd number of green edges and an odd number of blue edges, or
with an even number of red edges, an even number of green edges and
an even number of blue edges. Since the red edges form an acyclic
subgraph of $H=\GG''/F$, there are at most $\sum_{\ell=2}^\infty d_{\ell}-1$
red edges and thus the total weight $r$ of red edges
is at most $\sum_{\ell=2}^\infty d_{\ell}$ (we forget ``$-1$'' since
it is not important for our further estimates). A symmetric
argument yields that $g\le\sum_{\ell=2}^\infty d_{\ell}$.

Let us have a closer look at circuits of $F$ with weight two.
Such circuits correspond to pairs of parallel edges of $\GG''$ (and
thus of $\GG$). By our assumption, the only parallel edges contained
in $\GG$ are pairs of edges between two vertices $v_1$ and $v_2$ of degree three and
triples of edges between vertices $v_1$ and $v_2$ of degree three and five.
In the former case, both
$v_1$ and $v_2$ have degree three in $\GG''$. Consequently, each
of them is incident with a single colored edge. By the assumption
on the edge-coloring, the two edges have the same color. 
In the latter case, the third edge $v_1v_2$ which corresponds to a loop in $\GG''/F$
is blue. Hence, the other two edges incident with $v_2$
must have the same color, which is red, green or blue.
In both cases,
the vertex of $H$ corresponding to the circuit $v_1v_2$ is an isolated
vertex in the subgraph of $H$ formed by red edges or in the subgraph
formed by green edges (or both). It follows we can improve
the estimate on $r$ and $g$:
\begin{equation}
r+g\le2\sum_{\ell=2}^\infty d_{\ell}-d_2=d_2+2\sum_{\ell=3}^\infty d_{\ell}
\label{eq-mindegree-2}
\end{equation}

We are now ready to construct the cycle cover of the graph $\GG''$.
Its construction closely follows the one presented in the proof of
Theorem~\ref{thm-intermezzo}.
The cycle cover is formed by three cycles $\CC_1$, $\CC_2$ and $\CC_3$.
The cycle $\CC_1$ consists of all red and green edges and
the edges of $C(\RE\cup\GE)^A$ for every circuit $C$ of $F$.
$\CC_2$ consists of all red and green edges and
the edges of $C(\RE\cup\GE)^B$ for every circuit $C$.
Finally, the cycle $\CC_3$ contains all red and blue edges and
the edges of $C(\RE\cup\BE)^A$ for every circuit $C$.
Clearly, the sets $\CC_1$, $\CC_2$ and $\CC_3$ are cycles of $\GG''$ and
correspond to cycles of $\GG$ whose length is equal to the the weight
of the cycles $\CC_1$, $\CC_2$ and $\CC_3$ in $\GG''$.

We now estimate the total weight of the cycles $\CC_1$, $\CC_2$ and $\CC_3$.
Each red edge is covered three times, each green edge twice and
each blue edge once. Each edge of $F$ is contained
in either $\CC_1$ or $\CC_2$ and
for every circuit $C$ of $F$ at most half of its edges
are also contained in $\CC_3$. 
We conclude that the total length of the cycles $\CC_1$, $\CC_2$ and
$\CC_3$ can be bounded as follows (note that we apply (\ref{eq-mindegree-2}) to estimate the
sum $r+g$ and we also use the fact that the number of the edges of $F$
is at most $2m/3$ by Lemma~\ref{lm-rainbow-mindegree}):
\begin{eqnarray}
&& 3r+2g+b+\sum_{\ell=2}^\infty\left\lfloor\frac{3\ell}{2}\right\rfloor d_{\ell} \nonumber\\
&=& m+2r+g+\sum_{\ell=2}^\infty\left\lfloor\frac{\ell}{2}\right\rfloor d_{\ell} \nonumber\\
&\le& m-d_2+\sum_{\ell=2}^\infty\left(\left\lfloor\frac{\ell}{2}\right\rfloor+3\right) d_{\ell} \nonumber\\
&=& \frac{13m}{8}-\frac{5(r_0+g_0+b_0)}{8}-d_2+\sum_{\ell=2}^\infty\left(\left\lfloor\frac{\ell}{2}\right\rfloor+3-\frac{5\ell}{8}\right) d_{\ell} \nonumber\\
&\le& \frac{43m}{24}-\frac{5(r_0+g_0+b_0)}{8}-d_2+\sum_{\ell=2}^\infty\left(\left\lfloor\frac{\ell}{2}\right\rfloor+3-\frac{7\ell}{8}\right) d_{\ell} \nonumber\\
&\le& \frac{43m}{24}-\frac{5(r_0+g_0+b_0)}{8}-d_2+\sum_{\ell=2}^{6}\left(\left\lfloor\frac{\ell}{2}\right\rfloor+3-\frac{7\ell}{8}\right) d_{\ell} \nonumber\\
&=& \frac{43m}{24}-\frac{5(r_0+g_0+b_0)}{8}+5d_2/4+11d_3/8+3d_4/2+5d_5/8+3d_6/4\;\mbox{.}\label{eq-mindegree-LC}
\end{eqnarray}
The last inequality follows from the fact that $\left\lfloor\frac{\ell}{2}\right\rfloor+3-\frac{7\ell}{8}\le 0$
for $\ell\ge 7$.

The length of the shortest cycle cover of $\GG$ with three cycles
exceeds neither the bound
given in (\ref{eq-mindegree-SC}) nor the bound given in (\ref{eq-mindegree-LC}).
Hence, the length of the shortest cycle cover of $\GG$ is bounded by any
convex combination of the two bounds, in particular, by the following:
$$\frac{5}{9}\cdot\left(
    \frac{3m}{2}+\frac{r_0+g_0+b_0}{2}-d_2-3d_3/2-2d_4-d_5/2-d_6-3d_7/2
           \right) +$$
$$\frac{4}{9}\cdot\left(
    \frac{43m}{24}-\frac{5(r_0+g_0+b_0)}{8}+5d_2/4+11d_3/8+3d_4/2+5d_5/8+3d_6/4
           \right)=$$
$$\frac{44m}{27}-2d_3/9-4d_4/9-2d_6/9-5d_7/6
  \le\frac{44m}{27}\;\mbox{.}$$
The proof of Theorem~\ref{thm-main} is now completed.
\end{proof}

\section*{Acknowledgement}

This research was started while the second and the fifth authors
attended PIMS Workshop on Cycle Double Cover Conjecture held
at University of British Columbia, Vancouver, Canada. Support
of Pacific Institute for Mathematical Sciences (PIMS)
during the workshop is gratefully acknowledged.

\end{document}